\documentclass[twoside,11pt]{article}
\usepackage{amsthm}
\usepackage[8,preprint]{jmlr2e}
\usepackage{blindtext}
\RequirePackage{amsmath,amsfonts}
\usepackage{mathtools}
\usepackage{mathrsfs} 
\usepackage{dsfont}
\usepackage{svg}
\usepackage{subcaption}
\usepackage{algorithm}
\usepackage{algorithmic}
\usepackage{enumitem}
\usepackage{xcolor,todonotes}
\usepackage{ulem}

\newcommand{\cA}{\mathcal A}

\newcommand{\cY}{\mathcal Y}
\newcommand{\cT}{\mathcal T}
\newcommand{\R}{\mathds R}
\newcommand{\E}{\mathds E}

\newcommand{\bL}{\mathbf{L}}
\newcommand{\Var}{\mathbb{V}}

\newcommand{\F}{\mathcal{F}}

\newtheorem{assumption}{Assumption}

\newtheorem{apptheorem}{Theorem}[section]
\newtheorem{applemma}{Lemma}[section]
\newtheorem{appremark}{Remark}[section]

\usepackage{lastpage}
\jmlrheading{25}{2024}{1-\pageref{LastPage}}{~; Revised ~ }{~}{21-0000}{Bernard Bercu, Luis Fredes and Eméric Gbaguidi}

\ShortHeadings{On the SAGA algorithm with decreasing step}{Bercu, Fredes and Gbaguidi}
\firstpageno{1}

\begin{document}

\title{On the SAGA algorithm with decreasing step}


\author{\name Bernard Bercu \email bernard.bercu@math.u-bordeaux.fr \\
\name Luis Fredes \email luis.fredes@math.u-bordeaux.fr \\
\name Eméric Gbaguidi \email thierry-emeric.gbaguidi@math.u-bordeaux.fr\\
\addr Institut de Mathématiques de Bordeaux\\
       Université de Bordeaux\\
       UMR 5251, 351 cours de la libération, 33405 Talence, France}

\editor{~}

\maketitle

\begin{abstract}
Stochastic optimization naturally appear in many application areas, including machine
learning. Our goal is to go further in the analysis of the Stochastic Average Gradient Accelerated (SAGA) algorithm. To achieve this, we introduce a new  $\lambda$-SAGA algorithm which interpolates between the Stochastic Gradient Descent ($\lambda=0$) and the SAGA algorithm ($\lambda=1$). Firstly, we investigate the almost sure convergence of this new algorithm with decreasing step which allows us to avoid the restrictive strong convexity and Lipschitz gradient hypotheses associated to the objective function. Secondly, we establish a central limit theorem for the $\lambda$-SAGA algorithm. Finally, we provide the non-asymptotic $\bL^p$ rates of convergence.
\end{abstract}

\section{Introduction}
Our goal is to solve the classical optimization problem in $\R^d$ which can be written as
\begin{equation}\tag{$\mathcal{P}$}\label{problem1}
    \min_{x\in \R^d} f(x),
\end{equation}
where $f$ is the average of many functions,
\begin{equation}\label{problem1_eq1}
    f(x)=\dfrac{1}{N}\sum_{k=1}^N f_k(x).
\end{equation}
This type of problem is frequently encountered in statistical learning and a standard way to solve (\ref{problem1}) is to make use of the Gradient Descent algorithm. However, in a large context, this approach has a very high computational cost. This limitation has led to the development of many stochastic algorithms for optimization \citep{nguyen2018sgd,bottou2018optimization}.

These new methods have taken a major role in recent advances of the neural networks.
Our goal is to go further in the analysis of the Stochastic Gradient Descent (SGD) algorithm \citep{robbins1951stochastic} and the SAGA algorithm \citep{defazio2014saga}. The standard SGD algorithm is given for all $n\geqslant 1$, by
\begin{equation}\label{sgd}\tag{SGD}
    X_{n+1}=X_n - \gamma_n \nabla f_{U_{n+1}}(X_n)= X_n - \gamma_n (\nabla f (X_n) + \varepsilon_{n+1}),
\end{equation}
where the initial state $X_1$ is a squared integrable random vector of $\R^d$ which can be arbitrarily chosen, $\nabla f (X_n)$ is the gradient of the function $f$ calculated at the value $X_n$, 
$\varepsilon_{n+1}=\nabla f_{U_{n+1}} (X_n)- \nabla f (X_n)$ and $(U_n)$ is a sequence of independent and identically distributed random variables, with uniform distribution on $\{1,2,\dots,N\}$, which is also independent from the sequence $(X_n)$. Moreover, $(\gamma_n)$ is a positive deterministic sequence decreasing towards zero and satisfying the standard conditions
\begin{equation}\label{gamma_cond1}
    \sum_{n=1}^\infty \gamma_n = + \infty \qquad \text{and } \qquad \sum_{n=1}^\infty \gamma_n^2 < + \infty.
\end{equation}
We clearly have from (\ref{problem1_eq1}) that $(\varepsilon_{n})$ is a martingale difference sequence adapted to the filtration $(\F_n)$ where $\F_n=\sigma(X_1,\dots,X_n)$. \\
The SAGA algorithm is a stochastic variance reduction algorithm which was proposed ten years ago in the pioneering work of \cite{defazio2014saga}. It slightly differs from the SGD algorithm as it is given, for all $n\geqslant 1$, by 
\begin{equation}\label{saga}\tag{SAGA}
     X_{n+1}=X_n - \gamma_n \left(\nabla f_{U_{n+1}} (X_n) -g_{n,U_{n+1}} + \frac{1}{N}\sum_{k=1}^N g_{n,k}\right),
\end{equation}
where the initial states $X_0$ and $X_1$ are squared integrable random vectors of $\R^d$ which can be arbitrarily chosen, the initial value $g_{1,k}$ is given, for any $k=1,\dots,N$, by $g_{1,k}= \nabla f_k(X_0)$. Moreover, the sequence $(g_{n,k})$ is updated, for all $n\geqslant 1$ and $1\leqslant k\leqslant N$, as
\begin{equation}\label{def_gnk}
    g_{n+1,k}=\left\{ \begin{array}{ll}
     \nabla f_k(X_n) & \, \text{if } \, U_{n+1}=k, \vspace{0.2cm}\\
     g_{n,k} & \text{ otherwise}.
\end{array}\right.
\end{equation}

One can observe that in most of all papers dealing with the SAGA algorithm, the step size is a fixed value $\gamma$ which depends on the strong convexity constant $\mu$ and the Lipschitz gradient constant $L$ associated with $f$. This will not be the case here at all. Our work aims to investigate the almost sure convergence as well as the asymptotic normality of the SGD and SAGA algorithms with decreasing step sequence $(\gamma_n)$ satisfying (\ref{gamma_cond1}).

\vspace{0.5cm}
\noindent \textbf{Our contributions.}
\vspace{1ex}\\
The goal of this paper is to answer to several natural questions.
\begin{enumerate}[label=(\alph*)]
    \item Is it possible to study the convergence of the SAGA algorithm with decreasing step ?
    \item Can we relax the strong convexity and the Lipschitz gradient assumptions ?
    \item Can we prove a central limit theorem for our new version of the SAGA algorithm ?
     \item Is it possible to provide non-asymptotic $\bL^p$ bounds for the SAGA algorithm ?
\end{enumerate}
We shall propose positive answers to all these questions by extending \citep{defazio2014saga} in several directions. 

\vspace{0.5cm}
\noindent \textbf{Organization of the paper.}
\vspace{1ex}\\
The paper is structured as follows. Section \ref{sec:related_work} is devoted to the state of the art concerning the SGD and SAGA algorithms. In Section \ref{sec:framework}, we present our new version of the SAGA algorithm which we shall call the \ref{sagag} algorithm. Section \ref{sec:main_results} deals with the main results of the paper. We establish the asymptotic properties of our \ref{sagag} algorithm such as the almost sure convergence and the asymptotic normality. Non-asymptotic $\bL^p$ rates of convergence are also provided. In Section \ref{sec:experiments}, we illustrate our theoretical results by some numerical experiments on real dataset. All technical proofs are postponed to the appendices. 

\section{Related work} \label{sec:related_work}

The stochastic approximations, initiated by \cite{robbins1951stochastic} and \cite{kiefer1952stochastic}, have taken a major role in optimization issues. The SGD algorithm, often known as a special case of the Robbins-Monro algorithm, is probably the most standard stochastic algorithm used in machine learning. The properties of this algorithm were investigated in several studies. The almost sure convergence results were established in \citep{robbins1971convergence,bertsekas2000gradient,duflo1996algorithmes,kushner2003stochastic,roux2012stochastic,schmidt2017minimizing,bottou2018optimization}. The convergence rates were proven in \citep{kushner1979rates,pelletier1998almost,nguyen2018sgd,liu2022almost}.  
The study of the asymptotic normality of stochastic approximations also appear in several works such that \citep{sacks1958asymptotic,fabian1968asymptotic,duflo1996algorithmes,pelletier1998weak,zhang2016central}.

In a high-dimensional context, many accelerated algorithms were proposed in literature in order to improve the Robbins-Monro algorithm performances \citep{polyak1992acceleration,fercoq2016optimization,defazio2014saga,xiao2014proximal,allen2018katyusha,leluc2022sgd}. In this paper, we will focus on the SAGA algorithm first introduced by \cite{defazio2014saga} for the minimization of the average of many functions and which is a well-known variance reduction method. This algorithm is a variant of the Stochastic Average Gradient (SAG) method proposed earlier in \citep{roux2012stochastic,schmidt2017minimizing}. It uses the concept of covariates to make an unbiased variant of the SAG method that has similar performances but is easier to implement \citep{gower2020variance}. The idea behind the SAGA algorithm, is to make use of the control variates, a well-known technique in Monte-Carlo simulation designed to reduce the variance of the SGD algorithm in order to accelerate its convergence. This algorithm incorporates knowledge about gradients on all previous data points rather than only using the gradient for the sampled data point \citep{defazio2014saga,palaniappan2016stochastic}. This method requires a storage linear in $N$ \citep{gower2018tracking}. Several works have studied the convergence of the SAGA algorithm, which is undoubtedly one of the most celebrated variance reduction algorithms.

\cite{defazio2014saga} established that the SAGA algorithm converges in $\bL^2$ at exponential rate. This result has been shown by assuming that the function $f$ is $\mu$-strongly convex and with $L$-Lipschitz gradient and by considering a fixed constant step $\gamma$ which tightly depends on the unknown values $\mu$ and $L$.
The almost sure convergence of the SAGA algorithm was not investigated in \cite{defazio2014saga}.
More recently, it was shown by \cite{poon2018local} that for a fixed constant step $\gamma=1/(3L)$, $f(X_n)$ and $X_n$ both converge almost surely to $f(x^*)$ and $x^*$ respectively, where $x^*$ is the unique point of $\R^d$ such that $\nabla f(x^*)=0$. This algorithm has been also investigated in \citep{palaniappan2016stochastic,defazio2016simple,gower2018tracking,qian2019saga} and there are now many variations on the original SAGA algorithm of \cite{defazio2014saga}. For example, \cite{qian2019saga} proposed a variant of the SAGA algorithm that includes arbitrary importance sampling and minibatching schemes.

Despite a decade of research, several issues remain open on the SAGA algorithm. The choice of the step $\gamma_n$ is clearly one of them. The vast majority of the theory for the SAGA algorithm relies on a fixed constant step $\gamma$ depending on the values $\mu$ and $L$ \citep{defazio2014saga,defazio2016simple,palaniappan2016stochastic,gower2018tracking,poon2018local,gower2020variance}. However, from a practical point of view, the values $\mu$ and $L$ are unknown and there is no guarantee on the convergence results established for this algorithm. We shall propose here to make use of decreasing step sequence $(\gamma_n)$ which allows us to avoid these constraints and relax some classic assumptions such that the $\mu$-strong convexity. Moreover, to the best of our knowledge, no result about the asymptotic normality of the SAGA algorithm is available in the literature so far.

\section{The \ref{sagag} algorithm}\label{sec:framework}
We introduce in this section the \ref{sagag} algorithm which can be seen as a generalization of the SAGA algorithm. We recall below the general principle of the Monte Carlo method that gave birth to the \ref{sagag} algorithm. Suppose that we would like to estimate the expectation $\E[X]$ of a square integrable real random variable $X$. Let us also consider another square integrable real random variable $Y$ strongly positively correlated to $X$ and for which we know how to compute the expectation $\E[Y]$. Then, it is possible to find a reduced variance estimator of $\E[X]$, given by $Z_\lambda=X-\lambda (Y-\E[Y])$ with $\lambda$ in $[0,1]$ \citep{defazio2014saga,chatterji2018theory}. One can obviously see that $\E[Z_\lambda]=\E[X]$, which means that
$Z_\lambda$ is an unbiased estimator of $\E[X]$. Moreover, $\Var[Z_\lambda]=\Var[X]+\lambda^2 \Var[Y]-2\lambda \mathbb{C}ov(X,Y)$. Hence, as soon as $\mathbb{C}ov(X,Y)>0$, we can choose $\lambda$ in $[0,1]$ such that $\Var[Z_\lambda]\leqslant \Var[X]$.
Now, using this principle of variance reduction, the \ref{sagag} algorithm is defined, for all $n\geqslant 1$, by
\begin{equation}\label{sagag}\tag{$\lambda$-SAGA}
     X_{n+1}=X_n - \gamma_n \left(\nabla f_{U_{n+1}} (X_n) -\lambda\left(g_{n,U_{n+1}} - \frac{1}{N}\sum_{k=1}^N g_{n,k}\right)\right),
\end{equation}
where the initial states $X_0$ and $X_1$ are squared integrable random vectors of $\R^d$ which can be arbitrarily chosen, the parameter $\lambda$ belongs to $[0,1]$, and $(\gamma_n)$ is a positive deterministic sequence decreasing towards zero and satisfying (\ref{gamma_cond1}).

One can establish a link between the \ref{sgd}, \ref{saga} and \ref{sagag} algorithms. Indeed, the \ref{sagag} algorithm with $\lambda=0$ corresponds to the absence of variance reduction and reduces to the \ref{sgd} algorithm. Furthermore, one can easily see that we find again the \ref{saga} algorithm by choosing $\lambda=1$. The motivation to introduce and study the \ref{sagag} algorithm comes from our desire to propose a unified convergence analysis for the SGD and SAGA algorithms and to investigate what happens in the intermediate cases $0<\lambda<1$.
We shall now state the general assumptions which we will use in all the sequel.
\begin{assumption}\label{saga2_cond1}
    Assume that function $f$ is continuously differentiable with a unique equilibrium point $x^*$ in $\R^d$ such that $\nabla f(x^*) = 0$.
\end{assumption}
\begin{assumption}\label{saga2_cond2}
    Suppose that for all $x\in \R^d $ with $x\neq x^*$,
    \begin{equation*}
        \langle x-x^*,\nabla f(x)\rangle >0.
    \end{equation*}
\end{assumption}
\begin{assumption}\label{saga2_cond3}
    Assume there exists a positive constant $L$ such that, for all $x\in \R^d$,
\begin{equation*}
    \dfrac{1}{N}\sum_{k=1}^N \lVert \nabla f_{k} (x)-\nabla f_{k} (x^*) \rVert^2 \leqslant L\lVert x -x^*\rVert^2.
\end{equation*}
\end{assumption}

These assumptions are not really restrictive and they are fulfilled in many applications. One can observe that Assumption \ref{saga2_cond2} is obviously weaker than the standard hypothesis that each function $f_k$ for $1\leqslant k \leqslant N$ is $\mu$-strongly convex with $\mu>0$.
 Note also that Assumption \ref{saga2_cond3} ensures that at $x^*$, the gradient of all functions $f_k$ for any $1\leqslant k\leqslant N$, does not change arbitrarily with respect to the vector $x \in \R^d$. Such an assumption is essential for convergence of most gradient-based algorithms; without it, the gradient would not provide a good indicator of how far to move to decrease $f$. One can also observe that if each function $f_k$ has Lipschitz continuous gradient with constant $\sqrt{L_k}$, then Assumption \ref{saga2_cond3} is satisfied by taking $L$ as the average value of all $L_k$. The most interesting improvement here is that both conditions are local in $x^*$ and sufficient for all of our analysis. 

\section{Main results}\label{sec:main_results}
In this section, we present the main results of the paper. First of all, we provide an almost sure convergence analysis for the \ref{sagag} algorithm with decreasing step. After that, we establish its asymptotic normality.
Lastly, we conclude this section by focusing on non-asymptotic $\bL^p$ rates of convergence of this stochastic algorithm. 

\subsection{Almost sure convergence}
Our first result deals with the almost sure convergence of the \ref{sagag} algorithm. 

\begin{theorem}\label{sagag_th_convps}
    Consider a fixed $\lambda \in [0,1]$. Assume that $(X_n)$ is the sequence generated by the \ref{sagag} algorithm with decreasing step sequence $(\gamma_n)$ satisfying (\ref{gamma_cond1}). In addition, suppose that Assumptions \ref{saga2_cond1}, \ref{saga2_cond2} and \ref{saga2_cond3} are satisfied. Then, we have
    \begin{equation}\label{sagag_th_convps_res1}
        \lim_{n\to +\infty} X_n=x^* \qquad a.s.
    \end{equation}
   and  
   \begin{equation}\label{sagag_th_convps_res2}
        \lim_{n\to +\infty} f(X_n)=f(x^*) \qquad a.s.
    \end{equation}
\end{theorem}
\begin{proof} Recall that for all $n\geqslant 1$,
\begin{equation*}
     X_{n+1}=X_n - \gamma_n \left(\nabla f_{U_{n+1}} (X_n) -\lambda\left(g_{n,U_{n+1}} - \frac{1}{N}\sum_{k=1}^N g_{n,k}\right)\right).
\end{equation*}
Hence, the \ref{sagag} algorithm can be rewritten as
\begin{equation}\label{sagag_th_convps_eq0_a}
    X_{n+1}=X_n - \gamma_n \left(Y_{n+1} - \lambda Z_{n+1}\right),
\end{equation}
where 
\begin{equation*}
    \left\{ \begin{array}{rl}
         Y_{n+1}&=\nabla f_{U_{n+1}} (X_n), \vspace{0.3cm} \\
         Z_{n+1}&=\nabla f_{U_{n+1}} (\phi_{n,U_{n+1}}) -\dfrac{1}{N} {\displaystyle \sum_{k=1}^N} \nabla f_{k} (\phi_{n,k}),
    \end{array} \right.
\end{equation*}
and $\phi_{n,k}$ is the point such that $g_{n,k}= \nabla f_k (\phi_{n,k})$.
As $\F_n=\sigma(X_1,\dots,X_n)$ and the sequence $(U_n)$ is independent of the sequence $(X_n)$, we clearly have from (\ref{problem1_eq1}) that
\begin{equation}\label{sagag_th_convps_eq0_a2}
    \E[Y_{n+1}|\F_n]=\nabla f(X_n) \qquad \text{and} \qquad \E[Z_{n+1}|\F_n]=0 \qquad a.s.
\end{equation}
As a consequence, $(Z_n)$ is a martingale difference sequence adapted to the filtration $(\F_n)$. \\
Hereafter, define for all $n\geqslant 1$,
\[V_n=\lVert X_n-x^* \rVert^2.\]
We obtain from (\ref{sagag_th_convps_eq0_a}) that for all $n\geqslant 1$,
\begin{align*}
    V_{n+1}&=\lVert X_{n+1}-x^* \rVert^2,\\
    & =\lVert X_n-x^*  - \gamma_n (Y_{n+1} - \lambda Z_{n+1}) \rVert^2, \\
    &= V_n -2\gamma_n\langle  X_n-x^*, Y_{n+1} - \lambda Z_{n+1}\rangle +\gamma_n^2 \lVert Y_{n+1} - \lambda Z_{n+1}\rVert^2.
\end{align*}
Moreover, we have from Jensen's inequality and the fact that $\lambda$ belongs to $[0,1]$ that
\begin{equation}\label{sagag_th_convps_eq0_b}
    \begin{split}
      &\E[\lVert Y_{n+1} - \lambda Z_{n+1}\rVert^2| \F_n]\\
    &= \E[\lVert (Y_{n+1} - \nabla f_{U_{n+1}}(x^*))- \lambda (Z_{n+1}- \nabla f_{U_{n+1}}(x^*)) +(1-\lambda)\nabla f_{U_{n+1}}(x^*) \rVert^2| \F_n] \\
    &\leqslant 3\E[\lVert Y_{n+1} - \nabla f_{U_{n+1}}(x^*)\rVert^2| \F_n] +3\E[\lVert Z_{n+1}- \nabla f_{U_{n+1}}(x^*)\rVert^2| \F_n] \\
    &+3 \E[\lVert \nabla f_{U_{n+1}}(x^*)\rVert^2|\F_n].
    \end{split}
\end{equation}
First of all, we clearly have
\begin{equation}\label{sagag_th_convps_eq0_c}
    \E[\lVert \nabla f_{U_{n+1}}(x^*)\rVert^2|\F_n]=\dfrac{1}{N}\sum_{k=1}^N \lVert \nabla f_{k} (x^*)\rVert^2 \qquad a.s.
\end{equation}
In addition, denote
\[A_n=\dfrac{1}{N}\sum_{k=1}^N \lVert \nabla f_{k} (\phi_{n,k})-\nabla f_{k} (x^*) \rVert^2
\qquad \text{and} \qquad
\Sigma_n=\dfrac{1}{N} \sum_{k=1}^N \nabla f_{k} (\phi_{n,k}).
\]
Since $\nabla f(x^*)=0$, we obtain by expanding the norm that 
\begin{equation}\label{sagag_th_convps_eq0_d}
    \E[\lVert Z_{n+1}- \nabla f_{U_{n+1}}(x^*)\rVert^2| \F_n] 
 = A_n-\left\lVert \Sigma_n \right\rVert^2 \qquad a.s.
\end{equation}
Furthermore, define for all $x\in \R^d$,
\[\tau^2(x)=\dfrac{1}{N}\sum_{k=1}^N \lVert \nabla f_{k} (x)-\nabla f_{k} (x^*) \rVert^2.\]
One can observe that
\begin{equation}\label{sagag_th_convps_eq0_e}
    \E[\lVert Y_{n+1}- \nabla f_{U_{n+1}}(x^*)\rVert^2| \F_n] = \tau^2(X_n) \qquad a.s.
\end{equation}
Putting together the three contributions (\ref{sagag_th_convps_eq0_c}), (\ref{sagag_th_convps_eq0_d}) and (\ref{sagag_th_convps_eq0_e}), we deduce from (\ref{sagag_th_convps_eq0_b}) that 
\begin{equation}\label{sagag_th_convps_eq0_f}
     \E[\lVert Y_{n+1} -\lambda Z_{n+1}\rVert^2| \F_n] \leqslant 3(\tau^2(X_n) + A_n +\theta^*) \qquad a.s.
\end{equation}
where
\[\theta^*=\dfrac{1}{N}\sum_{k=1}^N \lVert \nabla f_{k} (x^*)\rVert^2.\]
Consequently, it follows from (\ref{sagag_th_convps_eq0_a2}) and (\ref{sagag_th_convps_eq0_f}) that for all $n\geq 1$,
\begin{equation}\label{sagag_th_convps_eq1}
    \E[V_{n+1}| \F_n]\leqslant V_n -2\gamma_n\langle  X_n-x^*, \nabla f(X_n)\rangle +3\gamma_n^2(\tau^2(X_n) + A_n +\theta^*) \quad a.s.
\end{equation}
Furthermore, let $(T_n)$ be the sequence of Lyapunov functions defined, for all $n\geqslant 2$, by
\begin{align}\label{sagag_lyap}
T_n&=V_n+3N\gamma_{n-1}^2  A_n.
\end{align}
It follows from the very definition of the sequence $(\phi_{n,k})$ associated with \eqref{def_gnk} that 
\begin{align}\label{sagag_th_convps_eq1b}
      \E[A_{n+1}| \F_n]&=\dfrac{1}{N}\sum_{k=1}^N \E[\lVert \nabla f_{k} (\phi_{n+1,k})-\nabla f_{k} (x^*) \rVert^2| \F_n],\nonumber\\
      &=\dfrac{1}{N}\sum_{k=1}^N \left(\dfrac{1}{N}\lVert \nabla f_{k} (X_n)-\nabla f_{k} (x^*) \rVert^2 + \left(1-\dfrac{1}{N}\right)\lVert \nabla f_{k} (\phi_{n,k})-\nabla f_{k} (x^*) \rVert^2 \right),\nonumber\\
      &= \dfrac{1}{N}\tau^2(X_n) +\left(1-\dfrac{1}{N}\right)A_n,
\end{align}
almost surely. Hence, we obtain from (\ref{sagag_th_convps_eq1}) and (\ref{sagag_th_convps_eq1b}) that
\begin{align}\label{sagag_th_convps_eq1c}
 \E[T_{n+1}| \F_n]
    &= \E[V_{n+1}| \F_n] +3N\gamma_{n}^2\E[A_{n+1}| \F_n], \nonumber\\
    &\leqslant V_n -2\gamma_n\langle  X_n-x^*, \nabla f(X_n)\rangle +3\gamma_n^2(\tau^2(X_n) +A_n+\theta^*) +3N\gamma_{n}^2\E[A_{n+1}| \F_n]\nonumber\\
    &=V_n+3N\gamma_n^2A_n-2\gamma_n\langle  X_n-x^*, \nabla f(X_n)\rangle+3\gamma_n^2(2 \tau^2(X_n) + \theta^*),\nonumber\\
    &\leqslant V_n+3N\gamma_{n-1}^2 A_n-2\gamma_n\langle  X_n-x^*, \nabla f(X_n)\rangle+3\gamma_n^2(2 \tau^2(X_n) + \theta^*), \nonumber\\
    &\leqslant T_n-2\gamma_n\langle  X_n-x^*, \nabla f(X_n)\rangle+3\gamma_n^2(2 \tau^2(X_n) + \theta^*). 
\end{align}
Additionally, we clearly have $V_n\leqslant T_n$ almost surely and it follows from Assumption \ref{saga2_cond3} that
\[\tau^2(X_n)\leqslant L V_n \leqslant L T_n.\]
Finally, we deduce from (\ref{sagag_th_convps_eq1c}) that
\begin{equation}\label{sagag_th_convps_eq2}
    \E[T_{n+1}| \F_n]\leqslant (1+6L\gamma_n^2)T_n-2\gamma_n\langle  X_n-x^*, \nabla f(X_n)\rangle + 3\gamma_n^2\theta^*  \qquad a.s.,
\end{equation}
which can be rewritten as 
\begin{equation*}\label{sagag_th_convps_eq2b}
    \E[T_{n+1}| \F_n]\leqslant (1+a_n)T_n+\mathcal{A}_n-\mathcal{B}_n \qquad a.s.
\end{equation*}
where $a_n=6L\gamma_n^2$, $\mathcal{A}_n=3\gamma_n^2\theta^*$ and $\mathcal{B}_n=2\gamma_n\langle  X_n-x^*, \nabla f(X_n)\rangle$.
The four sequences $(T_n)$, $(a_n)$, $(\mathcal{A}_n)$ and $(\mathcal{B}_n)$ are positive sequences of random variables adapted to $(\F_n)$. We clearly have from (\ref{gamma_cond1}) that
\[\sum_{n=1}^\infty a_n < +\infty \qquad \text{and} \qquad \sum_{n=1}^\infty \mathcal{A}_n < +\infty.\]
Then, it follows from the Robbins-Siegmund Theorem \citep{robbins1971convergence} given by Theorem \ref{thm_rs} that $(T_n)$ converges a.s. towards a finite random variable $T$ and the series
\begin{equation}\label{sagag_th_convps_eq2c}
    \sum_{n=1}^\infty \mathcal{B}_n <+\infty \qquad a.s.
\end{equation}
Consequently, $(V_n)$ also converges a.s. to a finite random variable $V$. It only remains to show that $V=0$ almost surely. Assume by contradiction that $V>0$. 
For some positive constants $a<b$, denote by $\Omega$ the annulus of $\R^d$, 
\[\Omega= \{x\in \mathbb{R}^d, \quad 0<a<\lVert x- x^* \rVert^2< b \}.\]
Let $F$ be the function defined, for all $x\in \R^d$, by
\[F(x)=\langle x - x^*,\nabla f(x)\rangle.\]
We have from Assumption \ref{saga2_cond1} that $F$ is a continuous function in $\Omega$ compact. It implies that there exists a positive constant $c$ such that $F(x)>c$ for all $x \in \Omega$. However, for $n$ large enough, $X_n\in \Omega$, which ensures that $\gamma_n\langle  X_n-x^*, \nabla f(X_n)\rangle>c\gamma_n$. Consequently, it follows from (\ref{sagag_th_convps_eq2c}) that
\[\sum_{n=1}^\infty \gamma_n <+\infty.\]
This is of course in contradiction with assumption (\ref{gamma_cond1}). Finally, we obtain that $V=0$ almost surely, leading to 
\[\lim_{n\to +\infty} X_n=x^* \qquad a.s.\]
By continuity of the function $f$, we also have (\ref{sagag_th_convps_res2}), which completes the proof of Theorem \ref{sagag_th_convps}.
\end{proof}
\subsection{Asymptotic normality}
We now focus our attention on the asymptotic normality of the \ref{sagag} algorithm with decreasing step. In this subsection, we assume that $f$ is twice differentiable and we denote by $H=\nabla^2 f(x^*)$ the Hessian matrix of $f$ at the point $x^*$. 

\begin{assumption} \label{saga_cond_eig}
Suppose that $f$ is twice differentiable with a unique equilibrium point 
$x^*$ in $\R^d$ such that $\nabla f(x^*) = 0$. Denote by $\rho =\lambda_{min}(H)$ the minimum eigenvalue of $H$. We assume that $\rho>1/2$.
\end{assumption}
\noindent The central limit theorem for the \ref{sagag} algorithm is as follows.
\begin{theorem}\label{sagag_th_tlc1}
    Consider a fixed $\lambda \in [0,1]$. Let $(X_n)$ be the sequence generated by the \ref{sagag} algorithm with decreasing step $\gamma_n=1/n$. Suppose that Assumption \ref{saga_cond_eig} is satisfied. Assume also that 
\begin{equation}
    \label{ascvgclt}
    \lim_{n \to +\infty} X_n=x^* \qquad a.s.
\end{equation}
Then, we have the asymptotic normality
    \begin{equation} 
    \label{cltsaga}
        \sqrt{n} (X_n -x^*)  \quad \overset{\mathcal{L}}{\underset{n\to +\infty}{\longrightarrow}} \quad \mathcal{N}_d(0,\Sigma)
    \end{equation}
    where the asymptotic covariance matrix is given by
    \[\Sigma=(1-\lambda)^2 \int_0^\infty (e^{-(H-\mathds{I}_d/2)u})^T \Gamma e^{-(H-\mathds{I}_d/2)u} du,\]
    with 
    \[\Gamma =\dfrac{1}{N}\sum_{k=1}^N \nabla f_{k} (x^*) \left(\nabla f_{k} (x^*)\right)^T.\]
\end{theorem}
\begin{proof}
    The proof of Theorem \ref{sagag_th_tlc1} can be found in Appendix \ref{app:sagag_th_tcl1}. 
\end{proof}

\setcounter{theorem}{0}
\begin{remark}\label{sagag_rmk_tlc1} It was proven in Theorem \ref{sagag_th_convps} that the almost sure convergence \eqref{ascvgclt} holds under Assumptions
\ref{saga2_cond1}, \ref{saga2_cond2}, \ref{saga2_cond3}. It is obvious to see that Assumption \ref{saga_cond_eig} implies Assumption \ref{saga2_cond1}.
Consequently, Theorem \ref{sagag_th_tlc1} is also true when replacing \eqref{ascvgclt} by Assumptions  \ref{saga2_cond2} and \ref{saga2_cond3}.
\end{remark}
Many conclusions can be drawn from Theorem \ref{sagag_th_tlc1}.
First of all, if we assume that $\Gamma$ is a positive definite matrix, we obtain that our \ref{sagag} algorithm with $0\leqslant \lambda <1$, converges towards a centered normal distribution with positive definite variance. However, as soon as $\lambda=1$,  the limit distribution becomes a centered normal with variance $\Sigma=0$, in other words a Dirac measure.
Thus, the asymptotic distribution of the \ref{saga} algorithm has zero variance and one can therefore try to understand it. In fact, the conditional variances of the two terms of the martingale difference ($\varepsilon_{n}$) extracted from this algorithm, converge almost surely to exactly the same matrix. Therefore, the conditional variance of $(\varepsilon_{n})$ vanishes which explains the final result for the SAGA algorithm.
Moreover, Theorem \ref{sagag_th_tlc1} clearly shows the asymptotic variance reduction effect. Indeed, when $\lambda$ grows to $1$, we observe that the variance $\Sigma$ decreases and converges towards 0. Hence, for statistical inference purposes such that hypothesis test and confidence interval, we can take $\lambda$ just a little smaller than 1 to reduce the variance with respect to \ref{sgd}, but without canceling it.

\subsection{Non-asymptotic convergence rates}
In the same vein as \cite{moulines2011non} for the Robbins-Monro algorithm, we shall now establish non-asymptotic $\bL^p$ convergence rates. Hence, our goal is to investigate, for all integer $p\geqslant 1$, the convergence rate of $\E[\lVert X_n-x^*\rVert^{2p}]$ for the \ref{sagag} algorithm where the decreasing step is defined, for all $n \geqslant1$ by,
\begin{equation}\label{gamma_cond2}
    \gamma_n=\dfrac{c}{n^\alpha},
\end{equation}
for some positive constant $c$ and $1/2<\alpha \leqslant 1$.
First of all, we focus our attention on the standard case $p=1$ by analyzing our algorithm with a little more stringent condition than Assumption \ref{saga2_cond2}.
\begin{assumption}\label{saga2_cond4}
    Assume there exists a positive constant $\mu$ such that for all $x\in \R^d $ with $x\neq x^*$,
    \begin{equation*}
        \langle x-x^*,\nabla f(x)\rangle \geqslant\mu\lVert x-x^* \rVert^2.
    \end{equation*}
\end{assumption}
Although this is a strengthened version of Assumption \ref{saga2_cond2}, it is still weaker than the usual $\mu-$strong convexity assumption on the function $f$. This condition is sometimes called in the literature the Restricted Secant Inequality. 

\setcounter{theorem}{2}
\begin{theorem}\label{sagag_th_mse1}
    Consider a fixed $\lambda \in [0,1]$. Let $(X_n)$ be the sequence generated by the \ref{sagag} algorithm with decreasing step sequence $(\gamma_n)$ defined by (\ref{gamma_cond2}). Suppose that Assumptions \ref{saga2_cond1}, \ref{saga2_cond3} and \ref{saga2_cond4} are satisfied with $2 c\mu\leqslant 2^\alpha$ and $2 c\mu>1$ if $\alpha=1$.
    Then, there exists a positive constant $K$ such that for all $n \geqslant 1$,
    \begin{equation}\label{sagag_th_mse1_res1}
        \E\big[\lVert X_n-x^*\rVert^2\big]\leqslant \dfrac{K}{n^\alpha}.
    \end{equation}
\end{theorem}
\begin{proof}
    The proof of Theorem \ref{sagag_th_mse1} can be found in Appendix \ref{app:sagag_th_mse1}. 
\end{proof}
Next, we carry out our analysis in the general case $p\geqslant1$. It requires a strengthened version of Assumption \ref{saga2_cond3} given as follows.
\begin{assumption}\label{sagag_cond1}
    Assume that for some integer $p\geqslant 1$, there exists a positive constant $L_p$ such that for all $x\in \R^d$,
\[\dfrac{1}{N}\sum_{k=1}^N \lVert \nabla f_{k} (x)-\nabla f_{k} (x^*) \rVert^{2p} \leqslant L_p \lVert x -x^*\rVert^{2p}.\]
\end{assumption}
\begin{theorem}\label{sagag_th_mse2}
    Consider a fixed $\lambda \in [0,1]$. Let $(X_n)$ be the sequence generated by the \ref{sagag} algorithm with decreasing step sequence $(\gamma_n)$ defined by (\ref{gamma_cond2}) and such that the initial state $X_1$ belongs to $\bL^{2p}$. Suppose that Assumptions \ref{saga2_cond1}, \ref{saga2_cond4} and \ref{sagag_cond1} are satisfied with $p c\mu\leqslant 2^\alpha$ and $c\mu>1$ if $\alpha=1$.
    Then, there exists a positive constant $K_p$ such that for all $n \geqslant 1$,
    \begin{equation}\label{sagag_th_mse2_res1}
        \E\big[\lVert X_n-x^*\rVert^{2p}\big]\leqslant \dfrac{K_p}{n^{p\alpha}}.
    \end{equation}
\end{theorem}

\begin{proof}
The proof of Theorem \ref{sagag_th_mse2} can be found in Appendix \ref{app:sagag_th_mse2}. 
\end{proof}
\begin{remark}\label{sagag_rmk_mse2}
 It is easy to see that Assumption \ref{sagag_cond1} also implies that for all $x\in \R^d$,
 \begin{equation}\label{sagag_rmk_mse2_res1}
     (f(x) -f(x^*))^p \leqslant \dfrac{\sqrt{L_p}}{2^p}\lVert x-x^*\rVert^{2p}.
 \end{equation}
 Then, it follows from Theorem \ref{sagag_th_mse2} together with inequality (\ref{sagag_rmk_mse2_res1}) that there exists a positive constant $M_p=2^{-p} K_p \sqrt{L_p}$ such that for all $n \geqslant 1$,
 \begin{equation}\label{sagag_rmk_mse2_res2}
     \E\big[(f(X_n) -f(x^*))^p\big] \leqslant \dfrac{M_p}{n^{p\alpha}}.
 \end{equation}
\end{remark}
\section{Numerical experiments}\label{sec:experiments}
Consider the logistic regression model
\citep{bach2014, bercu2020} associated with the classical minimization problem (\ref{problem1}) of the convex function $f$ given, for all $x\in \R^d$, by

\begin{equation*}\label{num_eq1}
    f(x)=\dfrac{1}{N}\sum_{k=1}^N f_k(x)=\dfrac{1}{N}\sum_{k=1}^N \left(
    \log (1+\exp (\langle x, w_k\rangle))-y_k\langle x, w_k\rangle\right),
\end{equation*}
where $x\!\in\!\R^d$ is a vector of unknown parameters, $w_k\!\in\! \R^d$ is a vector of features and the binary output $y_k\in\{0,1\}$. As stated, this problem is equivalent to the log-likelihood maximization problem, where the aim is to find the parameter $x^*$ that maximizes the probability of a given sample $((w_1,y_1), \ldots, (w_N,y_N))$, which follows a model depending only on the unknown parameter $x$. To be more precise, our model has a Bernoulli probability with parameter $p_k(x)$ following a logistic function for each $k$, 

\begin{equation*}
    p_k(x) = \frac{\exp(\langle x,w_k\rangle)}{1+\exp(\langle x,w_k\rangle)}.
\end{equation*}
It is easy to see that $f$ is twice differentiable and its Hessian matrix
is given by
\vspace{-1ex}
\begin{equation*}
\vspace{-1ex}
\nabla^2f(x)=\dfrac{1}{N}\sum_{k=1}^N p_k(x)(1-p_k(x)) w_k w_k^T.
\end{equation*}
Consequently, $f$ has an unique equilibrium point $x^*$ and if we assume that the minimum eigenvalue of $H=\nabla^2f(x^*)$ is greater than $1/2$, Assumption \ref{saga_cond_eig} will be automatically satisfied, and therefore Assumption \ref{saga2_cond1} too. Moreover, one can observe that Assumptions \ref{saga2_cond3} and \ref{sagag_cond1} hold with
\vspace{-1ex}
\begin{equation*}
\vspace{-1ex}
L_p=\dfrac{1}{4^p N} \sum_{k=1}^N ||w_k ||^{4p}.
\end{equation*}
We conducted experiments on the MNIST dataset in order to present a visualisation of the almost sure convergence in Theorem \ref{sagag_th_convps}, the asymptotic normality in Theorem \ref{sagag_th_tlc1} and the $\bL^2$ bound in Theorem \ref{sagag_th_mse1}. 
For the almost sure convergence, the training database considered here includes $N=60000$ images in gray-scale format and size $28\times28$. Each image $w_k$ is therefore a vector of dimension $d=28\times28=784$. Each of these images is identified with a number from 0 to 9 and we divide it into a binary classification so that $y_k=0$ if $\{0,1,2,3,4\}$ and $y_k=1$ if $\{5,6,7,8,9\}$.
The results concerning the convergence of the estimator $X_n$ of $x^*$ are illustrated in Figure \ref{fig:grad_cps250}. The convergence is ordered from slowest to fastest in an increasing order with respect to $\lambda \in \{0,0.5,0.9,1\}$.
\begin{figure}[H]
\centering
\includegraphics[width= 0.7\textwidth]{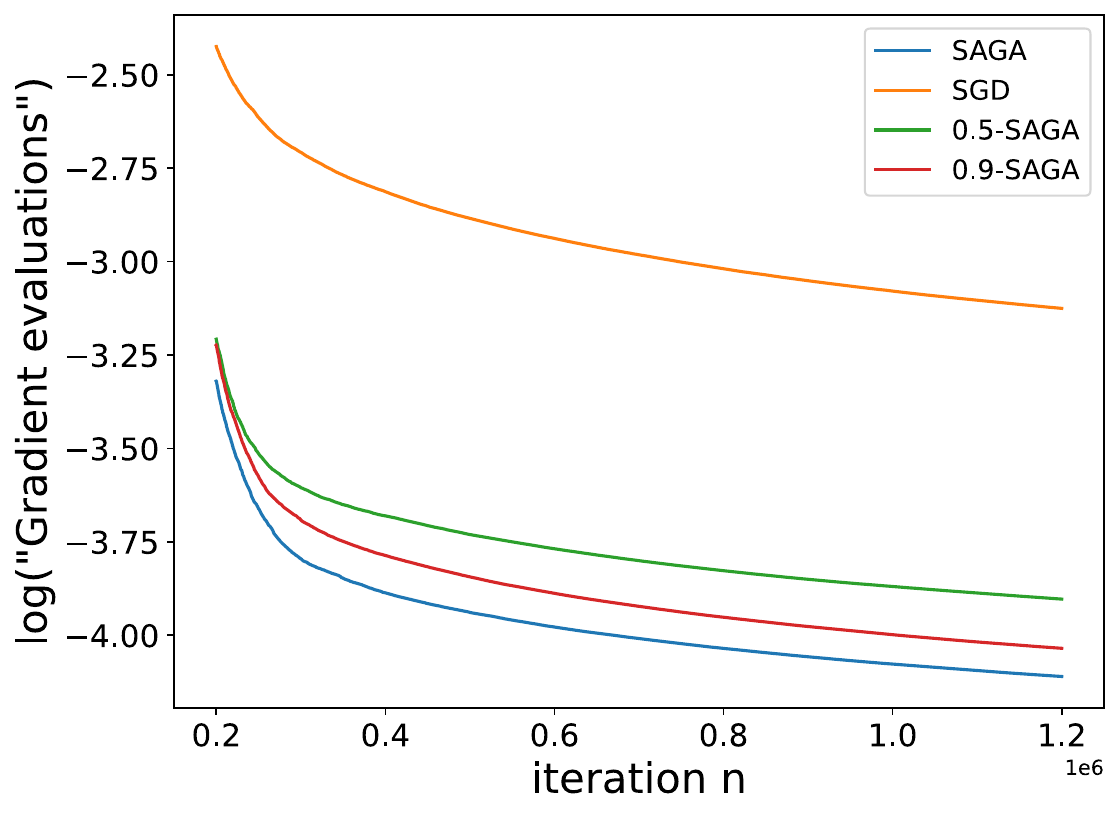}
\caption{Convergence with $\gamma_n=1/n$ for $1.2$M of iterations. Here we put ``Gradient evaluations'' since instead of using $\|\nabla f(X_n)\|$, we use the norm of the mean associated to the lines in the matrix $g_n$,  $\|\sum_{k=1}^N g_{n,k}\|/N$. This quantity keeps track of the convergence since it also converges to 0 and its lines converge to the gradients of the functions $f_k$, that is for each $1 \leqslant k \leqslant N$, $\lim g_{n,k}=\nabla f_k(x^*)$ as $n$ goes to infinity.}
\label{fig:grad_cps250}
\end{figure}
Moreover, to illustrate the asymptotic normality result, we use $N=100$, the first $100$ images in the MNIST dataset, and the distributional convergence 
\begin{equation*}
\lim_{n \rightarrow \infty} \E(h(\sqrt{n}(X_n-x^*))))= \E(h(\mathcal{N}_d(0,\Sigma))),
\end{equation*} 
where $h$ is defined, for all $x\in \R^d$, by
$h(x) = {\displaystyle \sum_{i=1}^d{x_i}}$.
It follows from Theorem \ref{sagag_th_tlc1} that
\begin{equation*}
h(\sqrt{n}(X_n-x^*))) \quad \overset{\mathcal{L}}{\underset{n\to +\infty}{\longrightarrow}} \quad \mathcal{N}(0,\sigma^2(\lambda)).
\end{equation*}
As the equilibrium point $x^*$ and the asymptotic variance $\sigma^2(\lambda)$ are unknown, we use estimators from the standard Monte Carlo procedure.
We denote for each fixed lambda $\widehat{\sigma}^{\,2}_n(\lambda)$ the estimator of $\sigma^2(\lambda)$. Given the form of our function $h$, we deduce that the limiting variances should be related as $(1-\lambda)^2\sigma^2(0) = \sigma^2(\lambda)$. The results are shown in Figure \ref{fig:all_graph_tlc}. 

\begin{figure}[H]
    \centering
    \begin{subfigure}[b]{0.49\textwidth}
        \includegraphics[width=\textwidth]{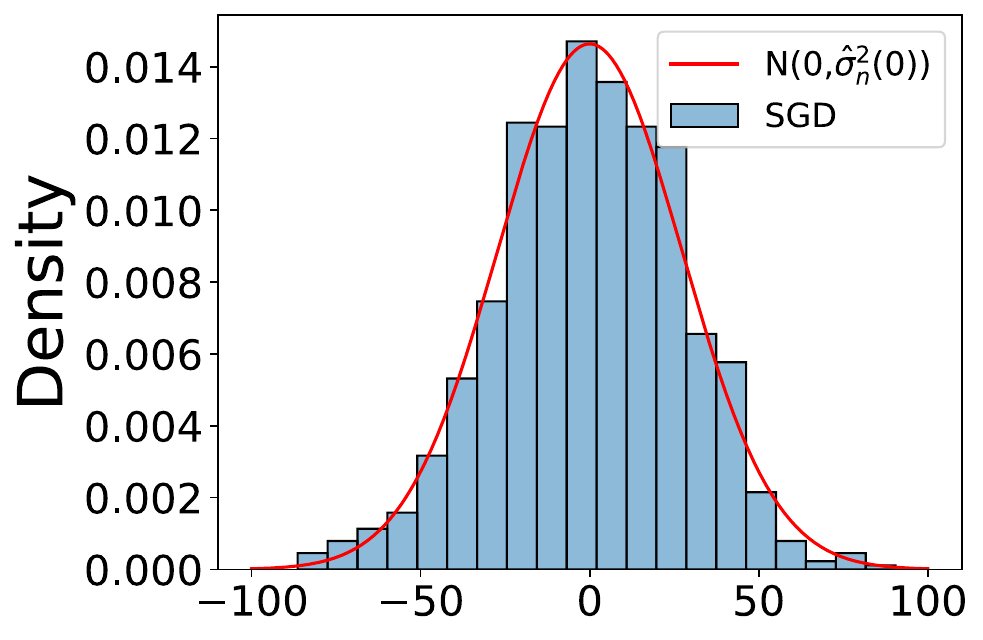}
        \caption{SGD $(\lambda = 0)$, where $\sigma_n^2(0)\approx 743.11$}
    \end{subfigure}
    \hfill
    \begin{subfigure}[b]{0.49\textwidth}
        \includegraphics[width=\textwidth]{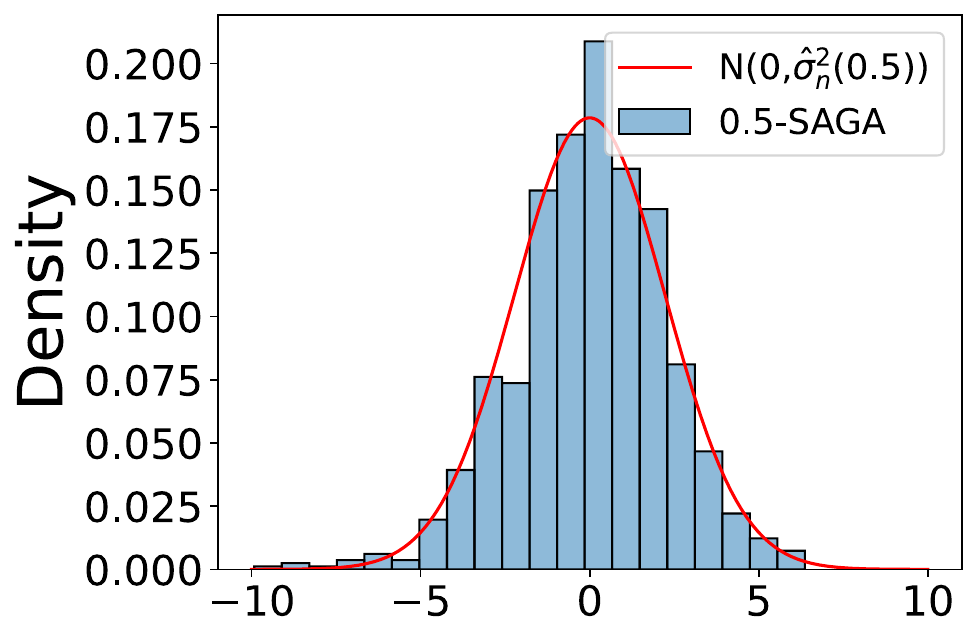}
        \caption{$0.5$-SAGA, where
        $\sigma_n^2(0.5)\approx 5$}
    \end{subfigure}
    \hfill
    \\
    \centering
    \begin{subfigure}[b]{0.49\textwidth}
        \includegraphics[width=\textwidth]{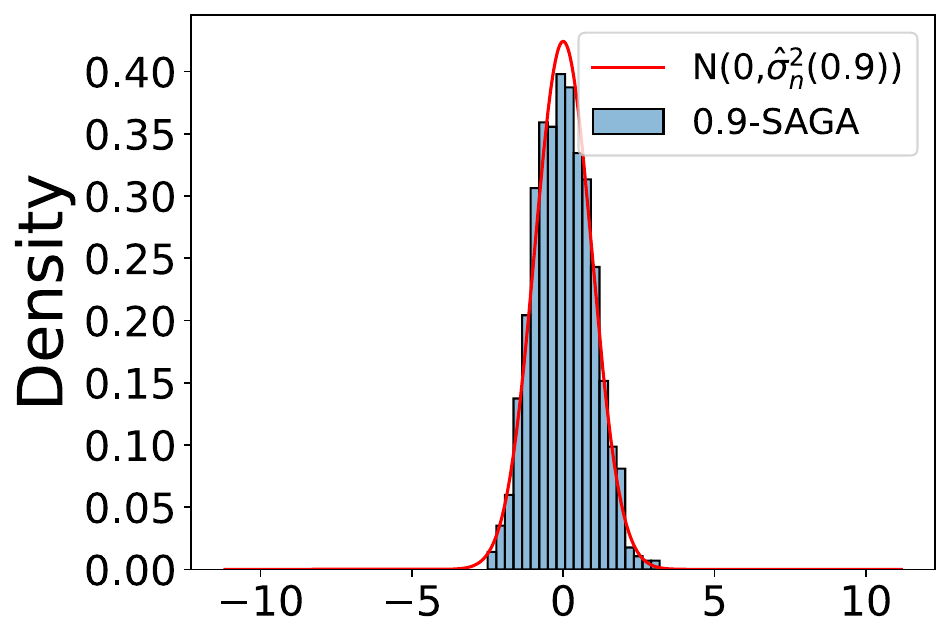}
        \caption{$0.9$-SAGA, where $\sigma_n^2(0.9)\approx 0.884$}
    \end{subfigure}
    \hfill
    \begin{subfigure}[b]{0.49\textwidth}
        \includegraphics[width=\textwidth]{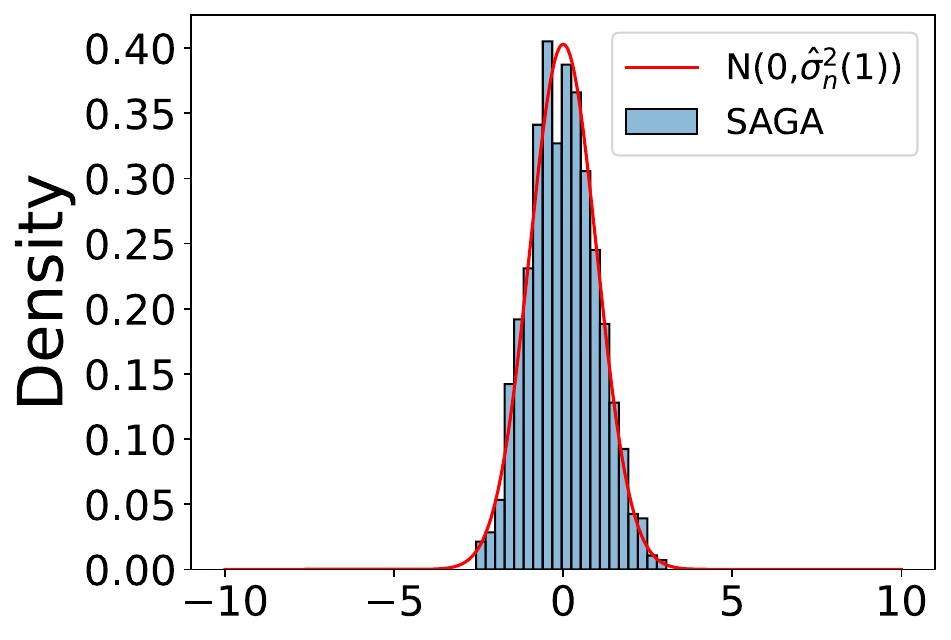}
        \caption{SAGA ($\lambda=1$), where $\sigma_n^2(1)\approx 0.98$}
    \end{subfigure}
    \hfill
    \caption{We used 1000 samples, where each one was obtained by running the associated algorithm for $n=500000$ iterations.}
    \label{fig:all_graph_tlc}
\end{figure}
\vspace{-3cm}
The main purpose of this plot is to represent the decreasing behavior of the variance with respect to the parameter $\lambda$. Even though for the SAGA ($\lambda=1$) we know that its variance converges to zero, for finite $n$ we can only see that it is shrinking with respect to $n$ to obtain at the limit a Dirac mass at 0. Here, the sample variances satisfy $\hat{\sigma}_n^2(0.9)<\hat{\sigma}_n^2(1)$. Nevertheless, they are still very close in the scale of the sample variance $\hat{\sigma}_n^2(0)$ of the SGD. We explain this as a consequence of the approximations and the fact that the models $\lambda=0.9$ and $\lambda=1$ are intimately related.

Finally, we present approximate results of the mean squared error $\E[\|X_n-x^*\|^2]$. For that purpose, we suppose that Assumption \ref{saga2_cond4} is satisfied so that Theorem \ref{sagag_th_mse1} holds. We run each algorithm for $100$ epochs, where each epoch consists of $1000$ iterations. In order to approximate the expectation, we apply the standard Monte Carlo procedure with $1000$ samples. Here, the approximation of $x^*$ is the result of running the SAGA algorithm for $40M$ iterations. The results are illustrated in Figure \ref{fig:L_p_norm_evolution}. The plot just give an intuition on the behavior of the mean squared error, since the constant $K$ in Theorem \ref{sagag_th_mse1} is unknown.
\begin{figure}[H]
\centering
\includegraphics[width= 0.7\textwidth]{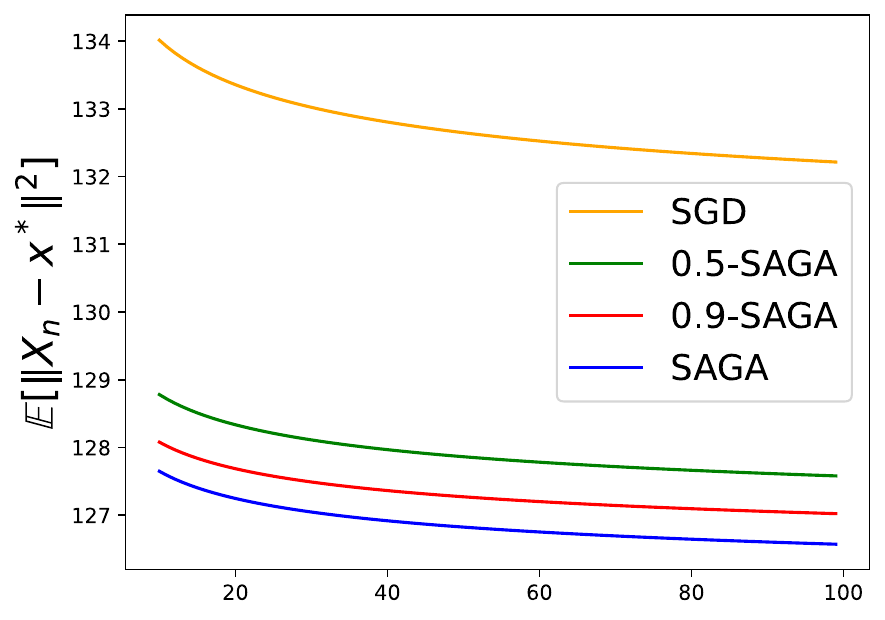}
\caption{Mean squared error with respect to epochs. We confirm the decreasing order of the mean squared error of $X_n-x^*$ with respect to $\lambda$ and $n$. }
\label{fig:L_p_norm_evolution}
\end{figure}

\section{Conclusion}
Stochastic optimization is one of the main challenges of machine learning touching almost every aspect of the discipline. Thus, in order to meet expectations, the SGD algorithm has been studied at length. However, the advent of Big Data for model learning led to the development of more sophisticated stochastic methods. In our study, we therefore highlight the properties of the new \ref{sagag} algorithm which is a generalization of the SAGA algorithm. We were able to establish the almost sure convergence and the asymptotic normality of this novel algorithm by using a decreasing step and without the strong convexity and Lipschitz gradient assumptions. The other major contribution of our paper concerns the convergence rates in $\mathbf{L}^p$ of the \ref{sagag} algorithm. Finally, stochastic algorithms offer multiple guarantees in terms of convergence and certainly promise to continue to have profound impacts on the fast development of the machine learning field.


\acks{This project has benefited from state support managed by the Agence Nationale de la Recherche (French National Research Agency) under the reference ANR-20-SFRI-0001.}

\vskip 0.2in

\newpage

\appendix

\section{Some useful existing results}
We first recall the well-known Robbins-Siegmund Theorem \citep{robbins1971convergence}.
\begin{apptheorem}[Robbins-Siegmund theorem]\label{thm_rs}
    Let $(V_n ) , (a_n ) , (\mathcal{A}_n ) , (\mathcal{B}_n)$ be four positive sequences of random variables adapted to a filtration $(\F_n )$ such that
    \[\E[V_{n+1}|\F_n]\leqslant (1+a_n) V_n +\mathcal{A}_n-\mathcal{B}_n,\]
    where
    \[\sum_{n=1}^{\infty} a_n < +\infty \qquad \text{and} \qquad \sum_{n=1}^{\infty} \mathcal{A}_n < +\infty \qquad a.s.
    \]
    Then, $(V_n)$ converges almost surely towards a finite random variable $V$ and 
    \[\sum_{n=1}^{\infty} \mathcal{B}_n < +\infty \qquad a.s.\]
\end{apptheorem}
\noindent The next two lemma provide very useful inequality for non-asymptotic convergence rates. The first lemma is given by Lemma A.3 in
supplementary material of \cite{bercu2021asymptotic}, see also Theorem 1 in \cite{moulines2011non}.

\begin{applemma}\label{app_lemma2} \citep{bercu2021asymptotic}. Let $(Z_n)$ be a sequence of positive real numbers satisfying, for all $n\geqslant 1$, the recursive inequality
\begin{equation}\label{app_lemma2_res1}
    Z_{n+1}\leqslant \left( 1- \dfrac{a}{(n+1)^\alpha}\right)Z_n +\dfrac{b}{(n+1)^\beta},
\end{equation}
where $a,b,\alpha$ and $\beta$ are positive constants satisfying $a\leqslant 2^\alpha$, $\alpha\leqslant 1$, $1< \beta < 2$ and $\beta\leqslant 2\alpha$ with $\beta< a+1$ in the special case where $\alpha=1$. Then, there exists a positive constant $C$ such that, for any $n\geqslant 1$,
\begin{equation}\label{app_lemma2_res2}
    Z_n \leqslant \dfrac{C}{n^{\beta-\alpha}}.
\end{equation}
\end{applemma}

\noindent
The second lemma is given without proof in \cite{chen2024online} in the special case $p\in (0,2]$, see Lemma B.3 as well as the seminal paper \cite{assouad1975}. We extend it to the case $p$ even and we propose a short 
proof for the sake of completeness.

\begin{applemma}\label{app_lemma1} Let $p$ be a positive even integer . It exist two positive constant $C_p$ and $D_p$ such that for any $a, \, b\in \R^d$,
\begin{equation}\label{app_lemma1_res1}
    \lVert a +b \rVert^{2+p} \leqslant \lVert a \rVert^{2+p} +(2+p)\langle a,b\rangle \lVert a \rVert^{p} + C_p\lVert a \rVert^{p} \lVert b \rVert^{2} +D_p \lVert b \rVert^{2+p}.
\end{equation}
\end{applemma}
\begin{proof}
    We prove Lemma \ref{app_lemma1} by induction. For the base case $p=2$, we have
    \begin{align*}
    \lVert a+b \rVert^4 &=\left(\lVert a\rVert^2 +2\langle a, b\rangle +\lVert b\rVert^2\right)^2\\
    &=\lVert a\rVert^4 +4(\langle a, b\rangle)^2 +\lVert b\rVert^4 + 4\langle a, b\rangle\lVert a\rVert^2+ 4\langle a, b\rangle\lVert b\rVert^2+2\lVert a\rVert^2\lVert b\rVert^2.
    \end{align*}
It follows from Cauchy–Schwarz inequality that $(\langle a, b\rangle)^2\leqslant \lVert a\rVert^2\lVert b\rVert^2$.
Moreover,  we also have $2\langle a, b\rangle \leqslant \lVert a\rVert^2+ \lVert b\rVert^2$ which implies that $4\langle a, b\rangle\lVert b\rVert^2\leqslant 2\lVert a\rVert^2\lVert b\rVert^2 +2\lVert b\rVert^4$.
Hence, we obtain from these two inequalities that
\[ \lVert a+b \rVert^4 \leqslant \lVert a\rVert^4 + 4\langle a, b\rangle\lVert a\rVert^2 +8\lVert a\rVert^2\lVert b\rVert^2+3\lVert b\rVert^4,\]
which leads to $C_2 = 8$ and $D_2=3$. Hereafter, assume that inequality \eqref{app_lemma1_res1} holds up to $q \geq 2$ and let $p=2+q$. We have by induction
    \begin{align*}
        &\lVert a+b\rVert^{2+p}= \lVert a+b\rVert^2\lVert a+b\rVert^p= \lVert a+b\rVert^2\lVert a+b\rVert^{2+q}\\
        &\leqslant  \Big(\lVert a\rVert^2 +2\langle a, b\rangle +\lVert b\rVert^2\Big) \Big( \lVert a \rVert^{2+q} +(2+q)\langle a,b\rangle \lVert a \rVert^{q} + C_{q}\lVert a \rVert^{q} \lVert b \rVert^{2} +D_{q} \lVert b \rVert^{2+q} \Big)\\
        &\leqslant \lVert a\rVert^{2+p} + (2+p)\langle a, b\rangle\lVert a \rVert^{p} + (C_{q}+2p+1)\lVert a\rVert^{p}\lVert b\rVert^{2} + (2C_{q}+ p)\|a\|^{q+1}\|b\|^3 \\
        &+ D_{q}\|a\|^2\|b\|^p + C_{q}\|a\|^{q}\|b\|^4  + 2D_{q}\|a\|\|b\|^{p+1} +D_{q}\lVert b\rVert^{2+p}
    \end{align*}
    where the last inequality is the result of applying Cauchy-Schwarz inequality $\langle a, b\rangle\leq \|a\|\|b\|$ for all the terms multiplied by $\langle a, b\rangle$ but the ones isolated in the second term. Furthermore, it follows from Young's inequality for products that
    \begin{align*}
        \|a\|^{q+1}\|b\|^3 = \|a\|^{q+1} \|b\|^{2(q+1)/p} \times\|b\|^{3 - 2(q+1)/p}  &\leq \frac{\|a\|^p\|b\|^2}{p/(p-1)} + \frac{\|b\|^{p+2}}{p},\\
        \|a\|^2\|b\|^p = \|a\|^2 \|b\|^{4/p} \times\|b\|^{p - 4/p}  &\leq \frac{\|a\|^p\|b\|^2}{p/2} + \frac{\|b\|^{p+2}}{p/q}, \\
        \|a\|^{q}\|b\|^4 = \|a\|^{q}\|b\|^{2q/p} \times \|b\|^{4-2q/p} &\leq \frac{\|a\|^p\|b\|^2}{p/q} + \frac{\|b\|^{p+2}}{p/2}, \\
        \|a\|\|b\|^{p+1}= \|a\|\|b\|^{2/p} \times \|b\|^{p+1-2/p} &\leq \frac{\|a\|^p\|b\|^2}{p} + \frac{\|b\|^{p+2}}{p/(p-1)}.
    \end{align*}
Finally, we obtain \eqref{app_lemma1_res1} with $C_p$ and $D_{p}$ satisfying the system defined, 
    for $p \geqslant 4$, by 
\[
\left \{
\begin{array}{ccc}
        C_{p} &=& \displaystyle{3p+\frac{4}{p}\Bigl( (p-1)C_{p-2} + D_{p-2}\Big)} \vspace{1ex}\\
        D_{p} &=& \displaystyle{1+\frac{4}{p}\Big(C_{p-2}+(p-1) D_{p-2}\Big)}
\end{array}
\right.
\]
with initial values $C_2 = 8$ and $D_2=3$, which achieves the proof of Lemma \ref{app_lemma1}.
\end{proof}
\vspace{-5ex}
\begin{appremark}
One can easily compute $C_4 = 39$ and $D_4 = 18$. Moreover, one can observe that we always have
$D_p \leqslant C_p$. Consequently, we can make use of \eqref{app_lemma1_res1} with $C_p$ instead of $D_p$.
\label{app_remark2}
\end{appremark}

\section{Proof of Theorem \ref{sagag_th_tlc1}}
\label{app:sagag_th_tcl1}
\begin{proof}
The \ref{sagag} algorithm can be rewritten as
\begin{equation*}
    X_{n+1}=X_n - \gamma_n (\nabla f(X_n) +\varepsilon_{n+1}),
\end{equation*}
where $\varepsilon_{n+1}=\cY_{n+1}-\lambda Z_{n+1}$ with
\begin{equation*}
    \left\{ \begin{array}{rl}
         \cY_{n+1}&=\nabla f_{U_{n+1}} (X_n)-\nabla f(X_n),\vspace{0.3cm}\\
         Z_{n+1}&=\nabla f_{U_{n+1}} (\phi_{n,U_{n+1}}) -\dfrac{1}{N} {\displaystyle \sum_{k=1}^N} \nabla f_{k} (\phi_{n,k}).
    \end{array} \right.
\end{equation*}
We already saw that $(\varepsilon_{n})$ is a martingale difference adapted to the filtration $(\F_n)$. Moreover, 
\begin{align*}
    \E[\varepsilon_{n+1} \varepsilon_{n+1}^T| \F_{n}]&=\E[\cY_{n+1} \cY_{n+1}^T| \F_{n}]-\lambda \E[\cY_{n+1} Z_{n+1}^T| \F_{n}]- \lambda \E[Z_{n+1} \cY_{n+1}^T| \F_{n}]\\
    &\hspace{0.3cm}+\lambda^2\E[Z_{n+1} Z_{n+1}^T| \F_{n}].
\end{align*}
In addition, we clearly have that almost surely
\begin{align*}
    \E[\cY_{n+1} \cY_{n+1}^T| \F_{n}]&= \dfrac{1}{N}\sum_{k=1}^N \nabla f_{k} (X_n) \left(\nabla f_{k} (X_n)\right)^T - \nabla f(X_n)\left(\nabla f(X_n)\right)^T,\\
    \E[Z_{n+1} Z_{n+1}^T| \F_{n}]&=\dfrac{1}{N}\sum_{k=1}^N \nabla f_{k} (\phi_{n,k}) \left(\nabla f_{k} (\phi_{n,k})\right)^T - \dfrac{1}{N^2}\left(\sum_{k=1}^N \nabla f_{k} (\phi_{n,k})\right)\left(\sum_{k=1}^N \left(\nabla f_{k} (\phi_{n,k})\right)^T\right),\\
    \E[Z_{n+1} \cY_{n+1}^T| \F_{n}]&=\dfrac{1}{N}\sum_{k=1}^N \nabla f_{k} (\phi_{n,k}) \left(\nabla f_{k} (X_n)\right)^T - \dfrac{1}{N}\sum_{k=1}^N \nabla f_{k} (\phi_{n,k}) \left(\nabla f(X_n)\right)^T.
\end{align*}
We now claim that for all $1\leqslant k\leqslant N$
\begin{equation}
\label{ascvpphink}
    \lim_{n\to +\infty} \phi_{n,k} = x^* \qquad a.s.
\end{equation}
As a matter of fact, for a fixed value $1\leqslant k\leqslant N$, the probability that $U_n=k$ occurs for infinitely many $n$. Consequently, $(\phi_{n,k})$ is a sub-sequence of $(X_n)$, since $\phi_{n,k}$ is updated to $X_n$ each time $\{U_n=k\}$. Hence, the almost sure convergence \eqref{ascvpphink} follows from \eqref{ascvgclt}.
Combining the almost sure convergence of $(X_n)$ and $(\phi_{n,k})$ towards $x^*$ with the continuity of $\nabla f$ given by Assumption \ref{saga_cond_eig}, it follows that almost surely
\begin{equation*}
    \left\{ \begin{array}{rl}
         {\displaystyle\lim_{n\to+\infty}} \E[\cY_{n+1} \cY_{n+1}^T| \F_{n}]&=\Gamma ,\vspace{0.2cm}\\
         {\displaystyle\lim_{n\to+\infty}} \E[\cY_{n+1} Z_{n+1}^T| \F_{n}]&=\Gamma ,\vspace{0.2cm}\\
         {\displaystyle\lim_{n\to+\infty}} \E[Z_{n+1} \cY_{n+1}^T| \F_{n}]&=\Gamma ,\vspace{0.2cm}\\
         {\displaystyle\lim_{n\to+\infty}} \E[Z_{n+1} Z_{n+1}^T| \F_{n}]&=\Gamma ,
    \end{array} \right. 
\end{equation*}
where
\[\Gamma =\dfrac{1}{N}\sum_{k=1}^N \nabla f_{k} (x^*) \left(\nabla f_{k} (x^*)\right)^T,\]
which leads to
\[\lim_{n \to + \infty}\E[\varepsilon_{n+1} \varepsilon_{n+1}^T\lvert \mathcal{F}_{n}]= (1-\lambda)^2 \Gamma \qquad a.s. \]
Therefore, we obtain from Toeplitz's lemma that
\[\lim_{n\to +\infty} \dfrac{1}{n}\sum_{k=1}^n \E[\varepsilon_k \varepsilon_k^T\lvert \mathcal{F}_{k-1}] = (1-\lambda)^2 \Gamma \qquad a.s.\]
In addition, we have for all $n\geqslant 1$
\begin{equation}\label{sagag_th_tlc1_eq1}
    \lVert\varepsilon_{n+1} \rVert \leqslant 2\left( \max_{k=1,\dots,N}  \lVert\nabla f_k (X_n) \rVert +\lambda \max_{k=1,\dots,N}  \lVert\nabla f_k (\phi_{n,k}) \rVert \right).
\end{equation}
Hence, it follows from (\ref{sagag_th_tlc1_eq1}) that 
\begin{equation}\label{sagag_th_tlc1_eq2}
    \lVert\varepsilon_{n+1} \rVert^4 \leqslant 128\left( \max_{k=1,\dots,N}  \lVert\nabla f_k (X_n) \rVert^4 +\lambda^4 \max_{k=1,\dots,N}  \lVert\nabla f_k (\phi_{n,k}) \rVert^4 \right).
\end{equation}
However, for all $k=1,...,N$, we have
\[\lVert\nabla f_k (X_n) \rVert^2 \leqslant 2 \lVert\nabla f_k (X_n) -\nabla f_k (x^*) \rVert^2 +2\lVert \nabla f_k (x^*)\lVert^2,\]
which implies that
\begin{equation*}
    \max_{k=1,\dots,N} \lVert\nabla f_k (X_n) \rVert^2 \leqslant 2N\Big( \tau^2(X_n)+\theta^*\Big).
\end{equation*}
Consequently,
\begin{equation}\label{sagag_th_tlc1_eq3}
    \max_{k=1,\dots,N}  \lVert\nabla f_k (X_n) \rVert^4 \leqslant 4N^2\Big(\tau^2(X_n)+\theta^*\Big)^2.
\end{equation}
By the same token,
\begin{equation}\label{sagag_th_tlc1_eq4}
    \max_{k=1,\dots,N} \lVert\nabla f_k (\phi_{n,k}) \rVert^4\leqslant 4N^2\Big(A_n+\theta^*\Big)^2.
\end{equation}
Hence, we obtain from (\ref{sagag_th_tlc1_eq2}), (\ref{sagag_th_tlc1_eq3}) and (\ref{sagag_th_tlc1_eq4}) that
\[\lVert\varepsilon_{n+1} \rVert^4 \leqslant 512N^2\Bigg(\Big(\tau^2(X_n)+\theta^*\Big)^2 +\lambda^4 \Big(A_n+\theta^*\Big)^2 \Bigg),\]
which immediately implies that
\begin{equation}\label{sagag_th_tlc1_eq5}
    \E[\lVert\varepsilon_{n+1} \rVert^4|\F_n] \leqslant 512N^2 \Bigg(\Big(\tau^2(X_n)+\theta^*\Big)^2 +\lambda^4 \Big(A_n+\theta^*\Big)^2 \Bigg) \qquad a.s.
\end{equation}
Moreover, since $X_n$ converges towards $x^*$, it follows that $\tau^2(X_n)$ converges to 0 almost surely. Combining this result with the almost sure convergence of $A_n$ towards 0 and (\ref{sagag_th_tlc1_eq5}), we find that 
\[\sup_{n\geqslant 1}  \E [\lVert\varepsilon_{n+1} \rVert^4 \lvert \F_{n}] < +\infty, \]
which implies that for all $\epsilon >0$,
\begin{equation*}
    \lim_{n\to +\infty}\dfrac{1}{n}\sum_{k=1}^n \E \left[\lVert\varepsilon_k\rVert^2 \mathds{1}_{\{ \lVert\varepsilon_k\rVert \geqslant \epsilon\sqrt{n}\}}\lvert \mathcal{F}_{k-1} \right]= 0 \qquad a.s.
\end{equation*}
Finally, it follows from the central limit theorem for stochastic algorithms given by Theorem 2.3 in \citep{zhang2016central} that 
\begin{equation*} 
        \sqrt{n} (X_n -x^*)  \quad \overset{\mathcal{L}}{\underset{n\to +\infty}{\longrightarrow}} \quad \mathcal{N}_d(0,\Sigma),
\end{equation*}
where 
\[\Sigma=(1-\lambda)^2 \int_0^\infty (e^{-(H-\mathds{I}_d/2)u})^T \Gamma e^{-(H-\mathds{I}_d/2)u} du,\]
which completes the proof of Theorem \ref{sagag_th_tlc1}.
 \end{proof}

\section{Proof of Theorem \ref{sagag_th_mse1}}
\label{app:sagag_th_mse1}
\begin{proof}
We already saw in \eqref{sagag_th_convps_eq1} that for all $n \geq 1$,
\[\E[V_{n+1}| \F_n]\leqslant V_n -2\gamma_n\langle  X_n-x^*, \nabla f(X_n)\rangle +3\gamma_n^2\big(\tau^2(X_n) + A_n +\theta^*\big) \qquad a.s.\]
Hence, it follows from Assumption \ref{saga2_cond4} that $\langle  X_n-x^*, \nabla f(X_n)\rangle \geqslant \mu V_n$, which leads to
\[\E[V_{n+1}| \F_n]\leqslant  (1-2\mu\gamma_n)V_n  +3\gamma_n^2\big(\tau^2(X_n) +A_n+\theta^*\big) \qquad a.s.\]
By taking the expectation on both side of this inequality, we obtain that for all $n \geq 1$,
\begin{equation}\label{sagag_th_mse1_eq1}
    \E[V_{n+1}]\leqslant  (1-2\mu\gamma_n) \E[V_n]  +3\gamma_n^2\big(\E[\tau^2(X_n)] +\E[A_n]+\theta^*\big).
\end{equation}
Furthermore, we deduce from Corollary \ref{sagag_cor_mse1a} in Appendix \ref{app:sagag_th_mse1a} below that there exist positive constants $b_1$ and $b_2$ such that, for all $n\geqslant 1$,  $\E[\tau^2(X_n)]\leqslant b_1$ and $\E[A_n]\leqslant b_2$. Consequently, (\ref{sagag_th_mse1_eq1}) immediately leads, for all $n\geqslant 1$, to
\[\E[V_{n+1}]\leqslant  \left(1-\frac{a}{(n+1)^\alpha} \right)\E[V_n]  + \frac{b}{(n+1)^{2\alpha}}\]
where $a=2 \mu c$ and $b=3c^24^{\alpha}(b_1 +b_2+\theta^*)$.
Finally, we can conclude from Lemma \ref{app_lemma2} that there exists a positive constant $K$ such that for any $n \geqslant 1$,
\[\E[\lVert X_n-x^*\rVert^2]\leqslant \dfrac{K}{n^\alpha},\]
which completes the proof of Theorem \ref{sagag_th_mse1}.
\end{proof}

\section{Proof of Theorem \ref{sagag_th_mse2}}
\label{app:sagag_th_mse2}
\begin{proof} 
First of all, Theorem \ref{sagag_th_mse2} follows from Theorem \ref{sagag_th_mse1} in the special case $p=1$. Hence, we are going to prove Theorem \ref{sagag_th_mse2} by induction on $n \geqslant 1$ for some integer $p\geqslant 2$ satisfying Assumption \ref{sagag_cond1}. As the initial state $X_1$ belongs to $\bL^{2p}$, the base case is immediately true. 
Next, assume by induction that for some integer $m$ which will be fixed soon, there exists a positive constant $K_p$ such that for all $n \leqslant m$,
\begin{equation}
\label{inductionhypn}
    \E[\lVert X_n-x^*\rVert^{2p}]\leqslant \dfrac{K_p}{n^{p\alpha}}.
\end{equation}
We have from (\ref{sagag_th_convps_eq0_a}) together with Lemma 
\ref{app_lemma1} that it exists a positive constant $C_{p}$
such that for all $n\geqslant 1$,
\begin{align*}
    V_{n+1}^{p}&=\lVert X_{n+1}-x^* \rVert^{2p},\\
    & =\lVert X_n-x^*  - \gamma_n (Y_{n+1} - \lambda Z_{n+1}) \rVert^{2p}, \\
    &\leqslant V_n^{p} -2p\gamma_n V_n^{p-1}\langle X_n-x^*, Y_{n+1} - \lambda Z_{n+1} \rangle +C_{p}\gamma_n^2V_n^{p-1}\lVert Y_{n+1} - \lambda Z_{n+1} \rVert^2 \\ &+C_{p}\gamma_n^{2p} \lVert Y_{n+1} - \lambda Z_{n+1} \rVert^{2p}.
\end{align*}
Hence, it follows from \eqref{sagag_th_convps_eq0_a2} that
\begin{align}
\label{saga_th_lpvpn}
    \E[V_{n+1}^{p}|\F_n]&\leqslant V_n^{p}-2p\gamma_n V_n^{p-1} \langle X_n-x^*, \nabla f (X_n) \rangle  \\
    &+C_{p}\gamma_n^2V_n^{p-1} \E[\lVert Y_{n+1} - \lambda Z_{n+1} \rVert^2|\F_n] +C_{p}\gamma_n^{2p} \E[\lVert Y_{n+1} - \lambda Z_{n+1} \rVert^{2p}|\F_n]. \nonumber
\end{align}
We already saw from \eqref{sagag_th_convps_eq0_f} that
\[\E[\lVert Y_{n+1} -\lambda Z_{n+1}\rVert^2| \F_n] \leqslant 3\big(\tau^2(X_n) + A_{n}+\theta^*\big) \qquad a.s.\]
which leads, via Assumption \ref{sagag_cond1}, to 
\[\E[\lVert Y_{n+1} -\lambda Z_{n+1}\rVert^2| \F_n] \leqslant 3
\big(L_p^{1/p} \, V_n +A_{n} +\theta^*\big) \qquad a.s.\]
Moreover, as in the proof of \eqref{sagag_th_convps_eq0_b}, we deduce from Jensen's inequality that
\begin{align}\label{sagag_th_lp}
   &\E[\lVert Y_{n+1} - \lambda Z_{n+1}\rVert^{2p}| \F_n]
    \leqslant 3^{2p-1} \E[\lVert Y_{n+1} - \nabla f_{U_{n+1}}(x^*)\rVert^{2p}| \F_n] \\
   & +3^{2p-1}\E[\lVert Z_{n+1}- \nabla f_{U_{n+1}}(x^*)\rVert^{2p}| \F_n] +
    3^{2p-1} \E[\lVert \nabla f_{U_{n+1}}(x^*)\rVert^{2p}|\F_n]. \nonumber
\end{align}
Hereafter, we clearly have
\begin{equation}\label{sagag_th_lpce1}
    \E[\lVert \nabla f_{U_{n+1}}(x^*)\rVert^{2p}|\F_n]=\dfrac{1}{N}\sum_{k=1}^N \lVert \nabla f_{k} (x^*)\rVert^{2p} \qquad a.s.
\end{equation}
Moreover, denote
\[A_{p,n}=\dfrac{1}{N}\sum_{k=1}^N \lVert \nabla f_{k} (\phi_{n,k})-\nabla f_{k} (x^*) \rVert^{2p}
\qquad \text{and} \qquad
\Sigma_n=\dfrac{1}{N} \sum_{k=1}^N \nabla f_{k} (\phi_{n,k}).
\]
It follows once again from Jensen's inequality that
\begin{equation}\label{sagag_th_lpce2}
    \E[\lVert Z_{n+1}- \nabla f_{U_{n+1}}(x^*)\rVert^{2p}| \F_n] 
 \leqslant 2^{2p-1} \big(A_{p,n}+\left\lVert \Sigma_n \right\rVert^{2p}\big) \qquad a.s.
\end{equation}
However, we obtain from Holder's inequality that
$\left\lVert \Sigma_n \right\rVert^{2p}\leqslant A_{p,n}$.
Consequently, inequality \eqref{sagag_th_lpce2} immediately leads to
\begin{equation}\label{sagag_th_lpce3}
    \E[\lVert Z_{n+1}- \nabla f_{U_{n+1}}(x^*)\rVert^{2p}| \F_n] 
 \leqslant 4^{p} A_{p,n} \qquad a.s.
\end{equation}
Furthermore, define for all $x\in \R^d$,
\[\tau^{2p}(x)=\dfrac{1}{N}\sum_{k=1}^N \lVert \nabla f_{k} (x)-\nabla f_{k} (x^*) \rVert^{2p}.\]
As in the proof of Theorem \ref{sagag_th_convps}, one can observe that
\begin{equation}\label{sagag_th_lpce4}
    \E[\lVert Y_{n+1}- \nabla f_{U_{n+1}}(x^*)\rVert^{2p}| \F_n] = \tau^{2p}(X_n) \qquad a.s.
\end{equation}
Putting together the three contributions (\ref{sagag_th_lpce1}), (\ref{sagag_th_lpce3}) and (\ref{sagag_th_lpce4}), we obtain from (\ref{sagag_th_lp}) that 
\begin{equation}\label{sagag_th_lpce5}
     \E[\lVert Y_{n+1} -\lambda Z_{n+1}\rVert^{2p}| \F_n] \leqslant 3^{2p-1}(\tau^{2p}(X_n) +4^{p} A_{p,n} +\theta^*_{p}) \qquad a.s.
\end{equation}
where
\[\theta^*_{p}=\dfrac{1}{N}\sum_{k=1}^N \lVert \nabla f_{k} (x^*)\rVert^{2p}.\]
Hence, Assumption \ref{sagag_cond1} implies that
\begin{equation}\label{sagag_th_lpce6}
     \E[\lVert Y_{n+1} -\lambda Z_{n+1}\rVert^{2p}| \F_n] \leqslant 3^{2p-1}(L_{p}V_n^{p} +4^{p} A_{p,n} +\theta^*_{p}) \qquad a.s.
\end{equation}
Therefore, we deduce from \eqref{saga_th_lpvpn} and \eqref{sagag_th_lpce6} that for all $n \geqslant 1$,
\begin{align*}
         \E[V_{n+1}^{p}|\F_n]&\leqslant \Big(1+3L_p^{1/p}\, C_{p} \gamma_n^2+3^{2p-1}L_{p}C_{p}\gamma_n^{2p}\Big) V_n^{p}\\
         &-2p\gamma_n V_n^{p-1}\langle X_n-x^*, \nabla f (X_n) \rangle  
    +3C_{p}\gamma_n^2V_n^{p-1}(A_{n} +\theta^*) \\
    &+3^{2p-1}C_{p}\gamma_n^{2p}(4^{p} A_{p,n} + \theta^*_{p}) \qquad a.s.
\end{align*}
Furthermore, it follows from Assumption \ref{saga2_cond4} that $\langle  X_n-x^*, \nabla f(X_n)\rangle \geqslant \mu V_n$, which leads to
\begin{align}
\label{sagag_th_lpce6b}
         \E[V_{n+1}^{p}|\F_n]&\leqslant \Big(1+3L_p^{1/p}\, C_{p} \gamma_n^2+3^{2p-1}L_{p}C_{p}\gamma_n^{2p}-2 \mu p\gamma_n\Big) V_n^{p}\nonumber\\
    &+3C_{p}\gamma_n^2V_n^{p-1}(A_{n} +\theta^*) 
    +3^{2p-1}C_{p}\gamma_n^{2p}(4^{p} A_{p,n} + \theta^*_{p}) \qquad a.s.
\end{align}
By taking the expectation on both side of this inequality, we obtain that for all $n \geq 1$,
\begin{align}
\label{sagag_th_lpce7}
         \E[V_{n+1}^{p}]&\leqslant \Big(1+3L_p^{1/p} \, C_{p} \gamma_n^2+3^{2p-1}L_{p}C_{p}\gamma_n^{2p}-2 \mu p\gamma_n\Big) \E[V_n^{p}] +3C_{p}\gamma_n^2\E[A_{n}V_n^{p-1}] \nonumber \\
    & +3C_{p}\gamma_n^2\theta^* \E[V_n^{p-1}]
    +3^{2p-1}C_{p}\gamma_n^{2p}(4^{p} \E[A_{p,n}] + \theta^*_{p}) 
\end{align}
We deduce from Corollary \ref{sagag_cor_mse2a} in Appendix \ref{app:sagag_th_mse2a} below that there exists a positive constant $d_p$ such that for all $n\geqslant 1$,  $\E[A_{p,n}]\leqslant d_p$. The main difficulty arising here is to find a sharp upper bound for the crossing term $\E[A_{n}V_n^{p-1}]$. By using once again Holder's inequality, we have for all $n \geqslant 1$,
\begin{equation}
\label{holdereav}
    \E[A_{n}V_n^{p-1}] \leq \Big( \E[A_{n}^{p}] \Big)^{\frac{1}{p}} \Big( \E[V_{n}^{p}] \Big)^{\frac{p-1}{p}}
    \qquad \text{and} \qquad
    \E[V_n^{p-1}] \leq  \Big( \E[V_{n}^{p}] \Big)^{\frac{p-1}{p}}.
\end{equation}
Hence, it is necessary to compute an upper bound for $\E[A_{n}^{p}]$. 
Nevertheless, one can observe from Jensen's inequality that we always have $A_{n}^{p} \leqslant A_{p,n}$, which leads that
$\E[A_{n}^{p}]\leqslant d_p$. Therefore, it follows from the induction hypothesis \eqref{inductionhypn} together with 
\eqref{sagag_th_lpce7} and \eqref{holdereav} that
\begin{equation}
\label{sagag_th_lpce8}
         \E[V_{n+1}^{p}]\leqslant \Big(1+\xi_n -2 \mu p\gamma_n\Big) \E[V_n^{p}] 
         +\frac{b}{(n+1)^{\alpha(p+1)}}
\end{equation}
where $\xi_n= 3 L_p^{1/p}\, C_p \gamma_n^2 + 3^{2p-1}L_{p}C_{p}\gamma_n^{2p}$ and 
$$
b=2^{\alpha( p+1)}c^2 C_p\Big( 3K_p^{(p-1)/p}\big(d_p^{1/p} +\theta^*\big)+c^{p-1}3^{2p-1}\big(4^pd_p+\theta_p^*\big)\Big).
$$
Hereafter, denote by $m$ the integer part of the real number
$$
\Bigl( \frac{3c C_p}{\mu p} \Bigr)^{1/\alpha} \Big(L_p^{1/p}+3^{2(p-1)}\Big)^{1/\alpha}.
$$
One can easily check that as soon as $n \geqslant m$, $\xi_n \leq \mu p \gamma_n$. Consequently, we find from
\eqref{sagag_th_lpce8} that as soon as $n\geqslant m$,
\begin{equation}
\label{sagag_th_lpce9}
         \E[V_{n+1}^{p}]\leqslant \left(1-\frac{a}{(n+1)^{\alpha}}\right) \E[V_n^{p}]
         +\frac{b}{(n+1)^{\alpha (p+1)}}
\end{equation}
where $a=p\mu c$. Finally, we deduce from Lemma \ref{app_lemma2} that there exists a positive constant $K_p$ such that for all $n \geqslant 1$,
\[\E[\lVert X_{n+1}-x^*\rVert^{2p}]\leqslant \dfrac{K_p}{(n+1)^{\alpha p}},\]
which achieves the induction on $n$ and completes the proof of Theorem \ref{sagag_th_mse2}.
\end{proof}
\section{Additional asymptotic result on the convergence in $\bL^2$}
\label{app:sagag_th_mse1a}
The goal of this appendix is to provide additional asymptotic properties
of the \ref{sagag} algorithm that will be useful in the proofs of our main results. First of all, we recall that $V_n=\lVert X_n-x^* \rVert^2$,
\begin{equation*}
    A_n=\dfrac{1}{N}{\displaystyle\sum_{k=1}^N} \lVert \nabla f_{k} (\phi_{n,k})-\nabla f_{k} (x^*) \rVert^2 \quad \text{and} \quad
    \tau^2(x)=\dfrac{1}{N}{\displaystyle\sum_{k=1}^N} \lVert \nabla f_{k} (x)-\nabla f_{k} (x^*) \rVert^2.
\end{equation*}

\begin{theorem}\label{sagag_th_mse1a}
    Consider a fixed $\lambda \in [0,1]$. Let $(X_n)$ be the sequence generated by the \ref{sagag} algorithm with decreasing step sequence $(\gamma_n)$ satisfying (\ref{gamma_cond1}). Suppose that Assumptions \ref{saga2_cond1}, \ref{saga2_cond3} and \ref{saga2_cond4} are satisfied.
Then, we have almost surely
\begin{equation}\label{sagag_th_mse1a_res1}
            \sum_{n=1}^{\infty} \gamma_n V_n < +\infty, \qquad
            \sum_{n=1}^{\infty} \gamma_n A_n < +\infty, \qquad
            \sum_{n=1}^{\infty} \gamma_n \tau^2(X_n) < +\infty.  
\end{equation}
In addition, we also have
\begin{equation}\label{sagag_th_mse1a_res2}
            \sum_{n=1}^{\infty} \gamma_n \E[V_n] < +\infty, \qquad
            \sum_{n=1}^{\infty} \gamma_n \E[A_n] < +\infty, \qquad
            \sum_{n=1}^{\infty} \gamma_n \E[\tau^2(X_n)] < +\infty.
\end{equation}
\end{theorem}
\begin{proof}
    Let us consider the same Lyapounov function used in the proof of Theorem \ref{sagag_th_convps} and given by \eqref{sagag_lyap}. We recall from inequality (\ref{sagag_th_convps_eq2}) that for all $n\geqslant 1$, 
\begin{equation}\label{sagag_th_mse1a_eq1}
    \E[T_{n+1}| \F_n]\leqslant (1+6L\gamma_n^2)T_n-2\gamma_n\langle  X_n-x^*, \nabla f(X_n)\rangle + 3\theta^* \gamma_n^2 \qquad a.s.
\end{equation}
Let $(\cT_n)$ be the sequence of Lyapunov functions defined, 
for all $n\geqslant 2$, by $\cT_n=b_nT_n $ where 
\begin{equation*}\label{sagag_th_mse1a_eq1a}
    b_n=\prod_{k=1}^{n-1} \big(1+6L\gamma_k^2\big)^{-1}.
\end{equation*}
Since $b_n=b_{n+1}(1+6L\gamma_n^2)$, we obtain from \eqref{sagag_th_mse1a_eq1} that for all $n\geqslant 1$,
\[ \E[\cT_{n+1}| \F_n]\leqslant \cT_n-2\gamma_n b_{n+1}\langle  X_n-x^*, \nabla f(X_n)\rangle + 3\theta^* b_{n+1}\gamma_n^2 \qquad a.s.\]
Hence, it follows from Assumption \ref{saga2_cond4} that for all $n\geqslant 1$,
\begin{equation}\label{sagag_th_mse1a_eq2}
    \E[\cT_{n+1}| \F_n]\leqslant \cT_n-2\mu\gamma_n b_{n+1} V_n+ 3\theta^* b_{n+1}\gamma_n^2 \qquad a.s.
\end{equation}
Moreover, we clearly have from the right-hand side of (\ref{gamma_cond1}) that $(b_n)$ converges to a positive real number $b$,
which implies that
\begin{equation}\label{sagag_th_mse1a_eq3a}
    \sum_{n=1}^{\infty} b_{n+1} \gamma_n^2 <+\infty.
\end{equation}
Therefore, we deduce from the Robbins-Siegmund Theorem \citep{robbins1971convergence} given by Theorem \ref{thm_rs} that $(\cT_n)$ converges almost surely towards a finite random variable $\cT$ and the series
\begin{equation*}\label{sagag_th_mse1a_eq3b}
    \sum_{n=2}^{\infty} \gamma_n b_{n+1} V_n < +\infty \qquad a.s.
\end{equation*}
which leads to
\begin{equation}\label{sagag_th_mse1a_eq3c}
    \sum_{n=1}^{\infty} \gamma_n V_n < +\infty \qquad a.s.
\end{equation}
We also obtain from relation \eqref{sagag_th_mse1a_eq3c} and Assumption \ref{saga2_cond3} that
\begin{equation}\label{sagag_th_mse1a_eq3d}
    \sum_{n=1}^{\infty} \gamma_n \tau^2(X_n) < +\infty \qquad a.s.
\end{equation}
In addition, by taking the expectation of both sides of (\ref{sagag_th_mse1a_eq2}) and using  a standard telescoping argument, we find that
\begin{equation}\label{sagag_th_mse1a_eq3e}
    2\mu\sum_{n=2}^{\infty} \gamma_n b_{n+1} \E[V_n] \leqslant  \E[\cT_{2}] + 3\theta^* \sum_{n=2}^{\infty} b_{n+1}\gamma_n^2.
\end{equation}
Then, it follows from \eqref{sagag_th_mse1a_eq3a} and \eqref{sagag_th_mse1a_eq3e} that
\begin{equation*}
    \sum_{n=2}^{\infty} \gamma_n b_{n+1} \E[V_n] < +\infty,
\end{equation*}
which implies that
\[\sum_{n=1}^{\infty} \gamma_n \E[V_n] < +\infty.\]
Consequently, we get from Assumption \ref{saga2_cond3} that
\begin{equation}
\label{sagag_th_mse1a_eq5}
\sum_{n=1}^{\infty} \gamma_n \E[\tau^2(X_n)] < +\infty .
\end{equation}
Furthermore, we already saw from \eqref{sagag_th_convps_eq1b} that for all $n\geqslant 1$,
\begin{equation}\label{sagag_th_mse1a_eq6}
    \E[A_{n+1}| \F_n]=\dfrac{1}{N}\tau^2(X_n) +\left(1-\dfrac{1}{N}\right)A_n \qquad a.s.
\end{equation}
For all $n\geqslant 1$, denote $\cA_n=\gamma_n A_n$.
Since $\gamma_{n+1}\leqslant \gamma_n$, we obtain from (\ref{sagag_th_mse1a_eq6}) that for all $n\geqslant 1$,
\begin{equation}\label{sagag_th_mse1a_eq7}
    \E[\cA_{n+1}| \F_n]\leqslant \cA_n+\dfrac{1}{N}\gamma_n\tau^2(X_n) -\dfrac{1}{N}\gamma_n A_n \qquad a.s.
\end{equation}
By considering the almost sure convergence \eqref{sagag_th_mse1a_eq3d}, it follows once again from the Robbins-Siegmund Theorem \citep{robbins1971convergence} given by Theorem \ref{thm_rs} that $(\cA_n)$ converges almost surely towards a finite random variable $\cA$ and the series
\[\sum_{n=1}^{\infty} \gamma_n A_n < +\infty \qquad a.s.\]
Moreover, by taking the expectation of both sides of (\ref{sagag_th_mse1a_eq7}) and using a standard telescoping argument, we obtain that
\begin{equation}
\label{sagag_th_mse1a_eq8}
    \sum_{n=1}^{\infty}\gamma_n\E[A_n] \leqslant N \E[\cA_1]+\sum_{n=1}^{\infty}\gamma_n\E[\tau^2(X_n)].
\end{equation}
Finally, we deduce from \eqref{sagag_th_mse1a_eq5} and \eqref{sagag_th_mse1a_eq8} that
\[\sum_{n=1}^{\infty}\gamma_n\E[A_n]<+\infty,\]
which completes the proof of Theorem \ref{sagag_th_mse1a}.
\end{proof}
\noindent A straightforward application of Theorem \ref{sagag_th_mse1a}, using the left-hand side of (\ref{gamma_cond1}), is as follows.
\begin{corollary}\label{sagag_cor_mse1a}
    Assume that the conditions of Theorem \ref{sagag_th_mse1a} hold. Then, we have 
    \[\lim_{n\to +\infty} V_n = \lim_{n\to +\infty} A_n=
    \lim_{n\to +\infty}  \tau^2(X_n)= 0 \qquad a.s.\]
Moreover, we also have
    \[\lim_{n\to +\infty} \E[V_n]= \lim_{n\to +\infty} \E[A_n]  =\lim_{n\to +\infty} \E[\tau^2(X_n)] =0.\]
\end{corollary}

\section{Additional asymptotic result on the convergence in $\bL^{p}$}
\label{app:sagag_th_mse2a}
As in the previous Appendix, our purpose is to establish additional asymptotic properties
of the \ref{sagag} algorithm that will be useful in the proofs of our main results. First of all, we recall that $V_n^p=\lVert X_n-x^* \rVert^{2p}$,
\begin{equation*}
    A_{p,n}=\dfrac{1}{N}{\displaystyle\sum_{k=1}^N} \lVert \nabla f_{k} (\phi_{n,k})-\nabla f_{k} (x^*) \rVert^{2p} \quad \text{and} \quad
    \tau^{2p}(x)=\dfrac{1}{N}{\displaystyle\sum_{k=1}^N} \lVert \nabla f_{k} (x)-\nabla f_{k} (x^*) \rVert^{2p}.
\end{equation*}
\begin{theorem}\label{sagag_th_mse2a}
    Consider a fixed $\lambda \in [0,1]$. Let $(X_n)$ be the sequence generated by the \ref{sagag} algorithm with decreasing step $\gamma_n$ satisfying (\ref{gamma_cond1}). Suppose that Assumptions \ref{saga2_cond1}, \ref{saga2_cond4} and \ref{sagag_cond1} are satisfied.
    Then, we have almost surely
\begin{equation}\label{sagag_th_mse2a_res1}
            \sum_{n=1}^{\infty} \gamma_n V_n^p< +\infty, \qquad
            \sum_{n=1}^{\infty} \gamma_n A_{p,n} < +\infty, \qquad
            \sum_{n=1}^{\infty} \gamma_n \tau^{2p}(X_n) < +\infty.  
\end{equation}
In addition, we also have
\begin{equation}\label{sagag_th_mse2a_res2}
            \sum_{n=1}^{\infty} \gamma_n \E[V_n^p] < +\infty, \qquad
            \sum_{n=1}^{\infty} \gamma_n \E[A_{p,n}] < +\infty, \qquad
            \sum_{n=1}^{\infty} \gamma_n \E[\tau^{2p}(X_n)] < +\infty.
\end{equation}
\end{theorem}
\begin{proof} We are going to prove Theorem \ref{sagag_th_mse2a} by induction on $p \geqslant 1$. First of all, Theorem \ref{sagag_th_mse2a} follows from Theorem \ref{sagag_th_mse1a} in the special case $p=1$. Hence, the base case is immediately true. 
Next, assume by induction that Theorem \ref{sagag_th_mse2a} holds for some integer $p-1$ with $p\geqslant 2$.
We recall from inequality (\ref{sagag_th_lpce6b}) in the proof of Theorem \ref{sagag_th_mse2} that for all $n\geqslant 1$,
\begin{align}
\label{sagag_th_mse2a_eq1}
         \E[V_{n+1}^{p}|\F_n]&\leqslant \Big(1+3L_p^{1/p}\, C_{p} \gamma_n^2+3^{2p-1}L_{p}C_{p}\gamma_n^{2p}-2 \mu p\gamma_n\Big) V_n^{p}\nonumber\\
    &+3C_{p}\gamma_n^2V_n^{p-1}(A_{n} +\theta^*) 
    +3^{2p-1}C_{p}\gamma_n^{2p}(4^{p} A_{p,n} + \theta^*_{p}) \qquad a.s.
\end{align}
However, it follows from Young's inequality for products that almost surely
\begin{equation} \label{sagag_th_mse2a_eq2}
    V_n^{p-1}A_{n}\leqslant \dfrac{A_n^p}{p}+ \dfrac{(p-1)V_n^p}{p}. 
\end{equation}
Moreover, one can observe from Jensen's inequality that $A_n^p\leqslant A_{p,n}$ almost surely. Combining the previous inequality with \eqref{sagag_th_mse2a_eq2}, it implies that
\begin{equation} \label{sagag_th_mse2a_eq3}
    V_n^{p-1}A_{n}\leqslant \dfrac{A_{p,n}}{p}+ \dfrac{(p-1)V_n^p}{p}. 
\end{equation}
Furthermore, by putting together the inequalities \eqref{sagag_th_mse2a_eq1} and \eqref{sagag_th_mse2a_eq3}, we obtain that
\begin{align}
\label{sagag_th_mse2a_eq4}
         \E[V_{n+1}^{p}|\F_n]&\leqslant \left(1+3\big(L_p^{1/p} +3p^{-1}(p-1)\big)C_{p} \gamma_n^2+3^{2p-1}L_{p}C_{p}\gamma_n^{2p}\right) V_n^{p} -2 \mu p\gamma_n V_n^{p}\nonumber\\
    &+3C_{p}\theta^*\gamma_n^2V_n^{p-1} 
    +\Big(3^{2p-1} 4^{p}+ 3p^{-1}\Big) C_{p}\gamma_n^{2}A_{p,n} +3^{2p-1}C_{p}\gamma_n^{2p}\theta^*_{p} \qquad a.s.
\end{align}
Let $(T_{p,n})$ be the sequence of Lyapunov functions defined, 
for all $n\geqslant 2$, by 
\begin{equation}\label{sagag_th_mse2a_eq5}
    T_{p,n}=V_n^p+ N e_p \gamma_{n-1}^{2}A_{p,n},
\end{equation}
where $e_p= 3^{2p-1}4^{p}+ 3p^{-1}$. By the definition \eqref{sagag_th_mse2a_eq5}, we have 
\begin{equation}\label{sagag_th_mse2a_eq6}
    \E[T_{p,n+1}|\F_n]=\E[V_{n+1}^{p}|\F_n]+ N e_p \gamma_{n}^{2}\E[A_{p,n+1}|\F_n] \qquad a.s.
\end{equation}
However, we deduce by the same arguments as in \eqref{sagag_th_convps_eq1b} that
\begin{equation}\label{sagag_th_mse2a_eq7}
      \E[A_{p,n+1}| \F_n]= \dfrac{1}{N}\tau^{2p}(X_n) +\left(1-\dfrac{1}{N}\right)A_{p,n} \qquad a.s.
\end{equation}
Hence, it follows from \eqref{sagag_th_mse2a_eq4}, \eqref{sagag_th_mse2a_eq6} and \eqref{sagag_th_mse2a_eq7} that
\begin{align}
\label{sagag_th_mse2a_eq8}
        \E[T_{p,n+1}|\F_n]&\leqslant T_{p,n} + 3C_p\left(\big(L_p^{1/p} +3p^{-1}(p-1)\big) \gamma_n^2+3^{2(p-1)}L_{p}\gamma_n^{2p}\right) V_n^{p} -2 \mu p\gamma_n V_n^{p} \nonumber\\
        &+ e_p\gamma_{n}^{2} \tau^{2p}(X_n) +3C_{p}\theta^*\gamma_n^2V_n^{p-1} 
    +3^{2p-1}C_{p}\gamma_n^{2p}\theta^*_{p} \qquad a.s.
\end{align}
Additionally, we clearly have $V_n^{p}\leqslant T_{p,n}$ and  Assumption \ref{sagag_cond1} leads to
\[\tau^{2p}(X_n)\leqslant L_p V_n^{p} \leqslant L_p T_{p,n}.\]
Finally, we obtain from \eqref{sagag_th_mse2a_eq8} that
\begin{equation}
    \label{sagag_th_mse2a_eq9}
        \E[T_{p,n+1}|\F_n]\leqslant (1+a_n) T_{p,n} +\Delta_n  -2 \mu p\gamma_n V_n^{p} 
        \qquad a.s.,
\end{equation}
where $\Delta_n=3C_{p}\theta^*\gamma_n^2V_n^{p-1} +3^{2p-1}C_{p}\gamma_n^{2p}\theta^*_{p}$ and 
\[a_n =  e_p L_p \gamma_n^2+ 3C_p\big(L_p^{1/p} +3p^{-1}(p-1) \big) \gamma_n^2 +3^{2p-1}L_{p}C_p \gamma_n^{2p}.\]
Since the sequence $(\gamma_n)$ satisfies (\ref{gamma_cond1}), it is easy to see that
\[\sum_{n=1}^\infty a_n <+\infty.\]
Moreover, by the induction hypothesis, we have that
\begin{equation}\label{sagag_th_mse2a_eq10}
    \left\{ \begin{array}{rl}
         &\displaystyle{\sum_{n=1}^\infty} \gamma_n V_n^{p-1} <+\infty \qquad a.s., \vspace{1ex}\\
         
         &\displaystyle{\sum_{n=1}^\infty} \gamma_n \E[V_n^{p-1}] <+\infty.
    \end{array} \right.
\end{equation}
From \eqref{sagag_th_mse2a_eq10}, one can immediately deduce that
\begin{equation}\label{sagag_th_mse2a_eq11}
    \left\{ \begin{array}{rl}
         &\displaystyle{\sum_{n=1}^\infty} \Delta_n <+\infty \qquad a.s., \vspace{1ex}\\
         &\displaystyle{\sum_{n=1}^\infty} \E[\Delta_n] <+\infty.
    \end{array} \right.
\end{equation}
Therefore, one uses exactly the same lines as in the proof of Theorem \ref{sagag_th_mse1a} and the Robbins-Siegmund Theorem \citep{robbins1971convergence} given by Theorem \ref{thm_rs} to show that
\begin{equation}\label{sagag_th_mse2a_eq12}
    \left\{ \begin{array}{rl}
         &\displaystyle{\sum_{n=1}^\infty} V_n^p <+\infty \qquad a.s., \vspace{1ex}\\
         &\displaystyle{\sum_{n=1}^\infty} \E[V_n^p] <+\infty.
    \end{array} \right.
\end{equation}
Hence, combining \eqref{sagag_th_mse2a_eq12} with Assumption \ref{sagag_cond1}, one immediately deduces that
\begin{equation}\label{sagag_th_mse2a_eq13}
    \left\{ \begin{array}{rl}
         &\displaystyle{\sum_{n=1}^\infty} \tau^{2p}(X_n) <+\infty \qquad a.s., \vspace{1ex}\\
         &\displaystyle{\sum_{n=1}^\infty} \E[\tau^{2p}(X_n)] <+\infty.
    \end{array} \right.
\end{equation}
Finally, using once again the same arguments as in the proof of Theorem \ref{sagag_th_mse1a}, we obtain that
\begin{equation}\label{sagag_th_mse2a_eq14}
    \left\{ \begin{array}{rl}
         &\displaystyle{\sum_{n=1}^\infty} A_{p,n} <+\infty \qquad a.s., \vspace{1ex}\\
         &\displaystyle{\sum_{n=1}^\infty} \E[A_{p,n}] <+\infty.
    \end{array} \right.,
\end{equation}
which achieves the proof of Theorem \ref{sagag_th_mse2a}.
\end{proof}
\noindent A useful consequence of Theorem \ref{sagag_th_mse2a}, using the left-hand side of (\ref{gamma_cond1}), is as follows.
\begin{corollary}\label{sagag_cor_mse2a}
    Assume that the conditions of Theorem \ref{sagag_th_mse2a} hold. Then, for all $p\geqslant 1$, we have 
    \[\lim_{n\to +\infty} V_n^p = \lim_{n\to +\infty} A_{p,n}=\lim_{n\to +\infty} \tau^{2p}(X_n)= 0 \qquad a.s.,\]
    and 
    \[\lim_{n\to +\infty} \E[V_n^p]=\lim_{n\to +\infty} \E[A_{p,n}]= \lim_{n\to +\infty} \E[\tau^{2p}(X_n)]  =0.\]
\end{corollary}

\end{document}